\documentclass[12pt, reqno]{amsart}
\usepackage[T1]{fontenc}
\usepackage[latin1]{inputenc}
\usepackage{typearea}
\usepackage{geometry}
\usepackage{ulem}
\usepackage{amsmath}
\usepackage{amsfonts}
\usepackage{amssymb}
\usepackage{latexsym}
\usepackage{enumerate}
\usepackage{verbatim}
\usepackage{amsthm}
\usepackage[all]{xy}
\usepackage{hhline}
\usepackage{epsf} 
\usepackage{cite}
\usepackage{mathrsfs}
\numberwithin{equation}{section}

\usepackage{color}

\newtheorem{theorem}{{\sc Theorem}}[section]
\newtheorem{cor}[theorem]{{\sc Corollary}}

\newtheorem{lemma}[theorem]{{\sc Lemma}}
\newtheorem{prop}[theorem]{{\sc Proposition}}

\theoremstyle{remark}
\newtheorem{remark}[theorem]{{\sc Remark}}

\theoremstyle{definition}
\newtheorem{de}[theorem]{\sc Definition}

\newcommand{\R}{\mathbb{R} }

\newcommand{\mfk}{\mathfrak{f}}
\newcommand{\mgk}{\mathfrak{g}}
\newcommand{\mhk}{\mathfrak{h}}

\newcommand{\N}{\mathbb{N} }

\newcommand{\F}{\mathcal{F}}

\newcommand{\Z}{\mathbb{Z}}
\newcommand{\calZ}{\mathcal{Z}}
\newcommand{\scrZ}{\mathscr{Z}}

\newcommand{\Prob}{\mathbb{P}}
\newcommand{\E}{\mathbb{E}}

\newcommand{\dk}{d_{\rm Kol}}

\newcommand{\Gammabar}{\bar{\Gamma}}

\providecommand{\abs}[1]{\lvert #1\rvert}
\providecommand{\babs}[1]{\bigl\lvert #1\bigr\rvert}
\providecommand{\Babs}[1]{\Bigl\lvert #1\Bigr\rvert}

\providecommand{\fnorm}[1]{\lVert #1\rVert_\infty}

\providecommand{\Enorm}[1]{\lVert #1\rVert_2}

\providecommand{\proj}[2]{\pro\bigl\{#1\,\bigl|\,C_{#2}\bigr\}}

\DeclareMathOperator{\Var}{Var}

\DeclareMathOperator{\dom}{dom}
\DeclareMathOperator{\Cov}{Cov}

\DeclareMathOperator{\pro}{proj}

\renewcommand{\phi}{\varphi}
\renewcommand{\epsilon}{\varepsilon}

\renewcommand{\rho}{\varrho}

\begin{document}

\title[The fourth moment theorem ]{The fourth moment theorem \\ on the Poisson space  }

\author{  Christian D\"obler \and Giovanni Peccati}
\thanks{\noindent Universit\'{e} du Luxembourg, Unit\'{e} de Recherche en Math\'{e}matiques \\
E-mails: christian.doebler@uni.lu, giovanni.peccati@gmail.com\\
{\it Keywords}: Poisson functionals; multiple Wiener-It\^{o} integrals; fourth moment theorem; carr\'{e}-du-champ operator; Berry-Esseen bounds; Gaussian approximation; Gamma approximation; Malliavin calculus; Stein's method\\
{\it AMS 2000 Classification}: 60F05; 60H07; 60H05}
\begin{abstract} We prove a fourth moment bound { without remainder} for the normal approximation of random variables belonging to the Wiener chaos of a general Poisson random measure. Such a result -- that has been elusive for several years -- shows that the so-called `fourth moment phenomenon', first discovered by Nualart and Peccati (2005) in the context of Gaussian fields, also systematically emerges in a Poisson framework. Our main findings are based on Stein's method, Malliavin calculus and Mecke-type formulae, as well as on a methodological breakthrough, consisting in the use of carr\'e-du-champ operators on the Poisson space for controlling residual terms associated with add-one cost operators. Our approach can be regarded as a successful application of Markov generator techniques to probabilistic approximations in a non-diffusive framework: as such, it represents a significant extension of the seminal contributions by Ledoux (2012) and Azmoodeh, Campese and Poly (2014). To demonstrate the flexibility of our results, we also provide some novel bounds for the Gamma approximation of non-linear functionals of a Poisson measure.
\end{abstract}

\maketitle

\section{Introduction}\label{intro}

\subsection{Overview} 

The aim of this paper is to prove a {\it fourth moment bound} { without remainder} for the normal approximation of random variables belonging to the Wiener chaos of a general Poisson measure. Differently from previous fourth moment limit theorems on the Poisson space proved in the literature, our main findings do not require that the involved random variables have the form of multiple integrals with a kernel of constant sign (see \cite{LRP1, Sch16, ET14}), nor that they are finite homogeneous sums (see \cite{PZ2}) or that they belong to Wiener chaoses of lower orders (see \cite{PT08, BPsv}). As discussed below, the methodological breakthrough yielding such an achievement, consists in the use of {\it carr\'e-du-champ operators} on the Poisson space, that we shall systematically exploit in connection with Mecke-type formulae and Stein's method  (see \cite{CGS, NouPecbook}). We will see that using carr\'e-du-champ operators instead of norms of Malliavin derivatives (as done in the already quoted references   \cite{LRP1, Sch16, ET14, PZ2, PT08, BPsv}) will allow us to bypass at once almost all combinatorial difficulties -- in particular connected to multiplication formulae on configuration spaces -- that have systematically marred previous attempts.

\smallskip

We stress that the idea of using carr\'e-du-champ operators, in order to deduce quantitative limit theorems by means of Stein's method, originates in the groundbreaking works \cite{Led12, ACP, AMMP}, where the authors apply the powerful techniques of {\it Gamma calculus} in the framework of eigenspaces of diffusive Markov generators (see \cite{BGL14} for definitions, { as well as \cite{ChPo} for an introduction to this approach}). As demonstrated in Section \ref{s:mwi}, our results show that such an approach can be fruitfully applied and extended, in order to control residual terms arising from the application of Stein's method in a non-diffusive context.

\subsection{Further historical details}
The so-called {\it fourth moment phenomenon} was first discovered in \cite{NP05}, where the authors  proved that a sequence of normalized random variables, belonging to a fixed Wiener chaos of a Gaussian field, converge in distribution to a Gaussian random variable if and only if their fourth cumulant converges to zero. Such a result constitutes a dramatic simplification of the method of moments and cumulants (see e.g. \cite[p. 202]{NouPecbook}), and represents a rough infinite-dimensional counterpart of classical results by de Jong -- see \cite{deJo87, deJo89, deJo90}, as well as \cite{DP16, DP16b} for recent advances. A particularly fruitful line of research was initiated in \cite{NP-ptrf}, where it is proved that the results of \cite{NP05} can be recovered from very general estimates, obtained by combining the Malliavin calculus of variations with Stein's method for normal approximation. {{} Precisely, one remarkable achievement of this approach is the bound 
\begin{equation}\label{gsbound}
d_{\rm Kol}(F,N)\leq \sqrt{\frac{q-1}{3q}\bigl(\E[F^4]-3\bigr)}\,, 
\end{equation}
where $d_{\rm Kol}$ stands for the Kolmogorov distance between the laws of two random variables, $F$ is a normalized \textit{multiple Wiener-It\^{o} integral} of order $q\geq1$ on a Gaussian space and $N$ denotes a standard normal random variable (see e.g. Theorem 5.2.6 in \cite{NouPecbook}, where analogous bounds for other metrics are also stated).}
Such a discovery has been the seed of a fruitful {{} stream} of research, now consisting of several hundred papers, where the results of \cite{NP05, NP-ptrf} have been extended and applied to a variety of frameworks, ranging from free probability to stochastic geometry, compressed sensing and time-series analysis  --- see the webpage {\tt https://sites.google.com/site/malliavinstein/home} for a constantly updated list, as well as the monograph \cite{NouPecbook} and the reference \cite{LNP} for recent developments related to functional inequalities.  

\smallskip

The line of research pursued in the present work stems from the two papers \cite{PSTU, PZ1}, where the authors adapted the techniques introduced in \cite{NP-ptrf} to the framework of non-linear Poisson functionals, in particular by combining Stein's method with a discrete version of Malliavin calculus on configuration spaces. {{} As anticipated, the principal aim of this work is to positively answer the following question: 
\begin{center}
\textit{Can one prove a bound comparable to \eqref{gsbound} on the Poisson space?} 
\end{center}
Such a question has stayed open since the publication of \cite{PSTU} and, so far,  answers have only been found in very special cases --- see Remark \ref{mtrem} below.}

\medskip

One should notice that the relevance of the techniques developed in \cite{PSTU, PZ1} has been greatly amplified by the pathbreaking reference \cite{RS}, where it is shown that one can use Malliavin-Stein techniques on the Poisson space in order to study the fluctuation of random objects arising in the context of random geometric structures on configuration spaces, like e.g. random graphs or random tessellations. Such a connection with stochastic geometry has generated a remarkable body of work, that has recently culminated in the publication of the monograph \cite{PecRei16}. The reader is referred to \cite{LPS, LSY16} for recent developments connected to Mehler formulae, stabilization and second order Poincar\'e inequalities, and to \cite{BacP} for some related concentration estimates in a geometric context.

 \subsection{Main results for normal approximations}  
\smallskip

We fix an arbitrary measurable space $(\mathcal{Z},\mathscr{Z})$ endowed with a $\sigma$-finite measure $\mu$. Furthermore, we let 
\[\mathscr{Z}_\mu:=\{B\in\mathscr{Z}\,:\,\mu(B)<\infty\}\]
and denote by 
\begin{equation*}
\eta=\{\eta(B)\,:\,B\in\mathscr{Z}\}
\end{equation*}
a \textit{Poisson measure} on $(\mathcal{Z},\mathscr{Z})$ with \textit{control} $\mu$, defined on a suitable probability space $(\Omega,\F,\Prob)$. We recall that the distribution of $\eta$ is completely determined by the following two facts: (i) 
{ for each finite sequence $B_1,\dotsc,B_m\in\mathscr{Z}$ of pairwise disjoint sets, the random variables $\eta(B_1),\dotsc,\eta(B_m)$ are independent,}
and (ii) that for every $B\in\mathscr{Z}$, the random variable $\eta(B)$ has the Poisson distribution with mean $\mu(B)$.
{ Here, we have extended the family of Poisson distributions to the parameter region $[0,+\infty]$ in the usual way.} 
For $B\in\mathscr{Z}_\mu$, we also write 
$\hat{\eta}(B):=\eta(B)-\mu(B)$ and denote by 
\[\hat{\eta}=\{\hat{\eta}(B)\,:\,B\in\mathscr{Z}_\mu\}\]
the \textit{compensated Poisson measure} associated with $\eta$. As discussed in Section \ref{ss:setup}, we require throughout the paper that $\eta$ is {\it proper}, that is, that $\eta$ can be a.s. represented as a (possibly infinite)
{ random} sum of Dirac masses. Without loss of generality, we may and will assume that $\F=\sigma(\eta)$. 
In order to state our main results, we introduce the following fundamental objects from stochastic analysis on the Poisson space. For precise definitions and further explanation we refer to \cite{PecRei16}, in particular to its first chapter \cite{Lastsv}, as well as to \cite{LPbook} and Section \ref{framework}. For a nonnegative integer $q$ and a square-integrable \textit{kernel function} $f\in L^2(\mu^q)$ we denote by $I_q(f)$ the $q$-th order \textit{multiple Wiener-It\^o integral} of $f$ with respect to $\widehat{\eta}$. If $L$ denotes the generator of the \textit{Ornstein-Uhlenbeck semigroup} with respect to $\eta$, then it is well-known that $-L$ has pure point spectrum given by the set of nonnegative integers and that, for 
$q\in\N_0=\{0,1,2,\dotsc\}$, $F$ is an eigenfunction of $-L$ with eigenvalue $q$, if and only if $F=I_q(f)$ for some $f\in L^2(\mu^q)$. The corresponding eigenspace $C_q$ is called the \textit{$q$-th Wiener chaos} associated with $\eta$.

\smallskip


Next, we introduce the probabilistic distances in which our bounds are expressed.
For $m\in\N$, denote by $\mathcal{H}_m$ the class of those $(m-1)$-times differentiable test functions 
$h$ on $\R$ such that $h^{(m-1)}$ is Lipschitz-continuous and we have   
\begin{equation*}
 \fnorm{h^{(l)}}\leq 1 \quad\text{for}\quad l=1,\dotsc,m\,.
\end{equation*}
{ Here and elsewhere, for an arbitrary function $g$ on $\R$, we use the notation 
\begin{equation*}
 \fnorm{g'}:=\sup_{x\not=y}\frac{\abs{g(y)-g(x)}}{\abs{y-x}}\in[0,+\infty]
\end{equation*}
for the \textit{minimum Lipschitz-constant} of $g$. This does not cause any confusion because this quantity coincides with the supremum norm of the derivative $g'$ of $g$ when $g$ happens to be differentiable.}
For real random variables $X$ and $Y$ such that $\E\abs{X},\E\abs{Y}<\infty$ we denote by 
\begin{equation*}
 d_m(X,Y):=d_m\bigl(\mathcal{L}(X),\mathcal{L}(Y)\bigr):=\sup_{h\in\mathcal{H}_m}\babs{\E[h(X)]-\E[h(Y)]}
\end{equation*}
the distance {between the distributions of $X$ and $Y$} induced by the class $\mathcal{H}_m$; observe that $d_1$ coincides with the classical 1-{\it Wasserstein distance}, see e.g. \cite[Appendix C]{NouPecbook} and the references therein. We will also study the {\it Kolmogorov distance} between the laws of $X$ and $Y$, given by
\begin{equation*}
 \dk(X,Y):=\sup_{x\in\R}\babs{\Prob[X\leq x ]-\Prob[Y\leq x]}.
\end{equation*}
{{} It is a well known fact (see e.g. \cite[Appendix C]{NouPecbook} and the references therein) that if $X$ is a generic random variable and $Y$ has a density bounded by $c\in(0,\infty)$, then
\begin{equation}\label{e:kw}
\dk(X,Y)\leq \sqrt{2c\; d_1(X,Y)}.
\end{equation}
}
\smallskip

The assumptions in our main results will be expressed in terms of the {\it add-one cost} operator $D^+$, that is defined as follows: if { $F = \mathfrak{f}(\eta)$} is a functional of $\eta$, then
$$
{ D^+_zF  := \mathfrak{f}(\eta+\delta_z)-\mathfrak{f}(\eta), \quad z\in \calZ,}
$$
in such a way that $D^+F$ can be regarded as a random function with domain equal to $\calZ$. See Section \ref{ss:setup} for a formal discussion of such an object.

{{} \begin{de}\label{d:H} { Let $F$ be an $\mathcal{F}$-measurable random variable (recall that $\mathcal{F}=\sigma(\eta)$).}
\begin{itemize}

\item[(i)]  We say that $F$ satisfies  {\bf Assumption A} if $F\in L^4(\Prob)$ and if the four random functions $D^+F$, $FD^+F$, $(D^+F)^4$ and $F^3 D^+F$ are all elements of the space 
$L^1 (\Omega\times Z, \F\otimes\mathscr{Z},\Prob\otimes \mu)=: L^1 (\Prob\otimes \mu)$.

\item[(ii)] We say that $F$ satisfies {\bf Assumption ${\bf A}^{\rm (loc)}$} if there exists a set $Z_0\in \mathscr{Z}$ with the following properties: (a) $\mu(\mathcal{Z} \backslash Z_0) = 0$, and (b) for every fixed $z\in Z_0$, the random variable $D^+_z F$ verifies {\bf Assumption A}.
\end{itemize}
\end{de}

{
\begin{remark} 
\begin{enumerate}[(a)]
\item Requiring that a given functionals satisfies {\bf Assumption A} or {\bf Assumption ${\bf A}^{\rm (loc)}$} is { necessary in this paper}, in order for us to apply Mecke{-type} identities (see \eqref{mecke}--\eqref{mecke2} below), as well as to exploit several almost sure representations of Malliavin and carr\'e-du-champ operators. { Both assumptions seem therefore to be (rather minimal) artefacts of the specific techniques adopted in the present paper, that are bound to be removed by further progress in the field. See \cite{DVZ} for several recent advances in this direction, partially building on the findings of the present paper, showing how {\bf Assumption A} can be avoided in the normal approximation of chaotic random variables in the Wasserstein distance, by using an adequate version of the exchangeable pairs approach of Stein's method. We also notice that a bound like \eqref{e:4mb4} on the Kolmogorov distance (which is of the same order as the corresponding bound on the Wasserstein distance), is for the time being outside the scope of exchangeable pairs.}
\item Using e.g. the multiplication formula stated in \cite[Proposition 5]{Lastsv}, one can easily prove that both {\bf Assumption ${\bf A}$} and {\bf Assumption ${\bf A}^{\rm (loc)}$} are verified, whenever $F$ has the form
$$
F = \sum_{q=0}^M I_q(f_q),
$$
where $M<\infty$ and each $f_q$ is bounded and such that its support is contained in a rectangle of the type $C\times \cdots \times C$, where $C\in \mathscr{Z}$ verifies $\mu(C)<\infty$. Such a class of random variables contains most $U$-statistics that are relevant for geometric applications (see the surveys \cite{LRRsv, STsv} and the references therein), as well as non-linear functionals of Volterra L\'evy processes \cite{PSTU, PZ1}, and the finite homogeneous sums in independent Poisson random variables considered in \cite{PZ2}. A similar remark applies to the assumptions appearing in the statement of our main abstract bounds in Proposition \ref{genbound} and Proposition \ref{p:genkol}. 
\end{enumerate}
\end{remark}
}
}

The next result is the main finding of the paper: it provides quantitative fourth moment estimates with completely explicit constants, both in the 1-Wasserstein and Kolmogorov distances, for random variables living in the Wiener chaos of a Poisson measure. Remarkably, the order of the bound (as a function of the fourth cumulant $\E[F^4]-3$) is the same for the two distances, thus significantly improving the estimate on $\dk$ that one could deduce from \eqref{e:kw}.

\begin{theorem}[Fourth moment bounds on the Poisson space]\label{4mt}
Fix an integer $q\geq 1$ and let $F = I_q(f)$ be a multiple Wiener-It\^{o} integral with respect to $\hat{\eta}$. Assume that $F$ verifies {\bf Assumption A} and that $\E[F^2]=1$; denote by $N\sim \mathscr{N}(0,1)$ a standard normal random variable. Then, 
\begin{align}
 d_1(F,N)&\leq \biggl(\sqrt{\frac{2}{\pi}}\frac{2q-1}{2q}+\frac{\sqrt{4q-3}}{\sqrt{q}}\biggr)\sqrt{\E\bigl[F^4\bigr]-3}\label{4mb}\\
&\leq\biggl(\sqrt{\frac{2}{\pi}}+2\biggr)\sqrt{\E\bigl[F^4\bigr]-3}\label{4mb2}\,
\end{align}
({in the above situation one automatically has that $\E[F^4]\geq3$}). Moreover, if, in addition to {\bf Assumption A}, $F$ also satisfies {\bf Assumption ${\bf A}^{\rm (loc)}$}, then
\begin{align}
\dk(F,N)&\leq \Bigl(11+2^{3/2}\bigl(\E[F^4]^{1/2}+\E[F^4]^{1/4}\bigr)\Bigr)\sqrt{\E\bigl[F^4\bigr]-3}\label{e:4mb3}
\quad\text{and}\\
\dk(F,N)&\leq  15.6 \sqrt{\E\bigl[F^4\bigr]-3} \label{e:4mb4}\,.
\end{align}
\end{theorem}

The following result is an immediate consequence of the bound \eqref{4mb2}.
\begin{cor}[Fourth moment theorem on the Poisson space]\label{cormt}
For each $n\in\N$ let $q_n\geq1$ be an integer and let $F_n=I_{q_n}(f_n)$ be a multiple Wiener-It\^{o} integral of some symmetric kernel $f_n\in L^2(\mu^{q_n})$ such that 
\[\lim_{n\to\infty}\E\bigl[F_n^2\bigr]=\lim_{n\to\infty}q_n!\Enorm{f_n}^2=1\quad\text{and}\quad
\lim_{n\to\infty}\E\bigl[F_n^4\bigr]=3\,.\]
Then, if each $F_n$ satisfies {\bf Assumption A}, the sequence $(F_n)_{n\in\N}$ converges in distribution to a standard normal random variable $N$ in the sense of the $1$-Wasserstein distance.
\end{cor}

\begin{remark}\label{mtrem}
\begin{enumerate}[(a)]
\item As mentioned before, so far, the fourth moment theorem on the Poisson space has only been known in very special cases: for double integrals, i.e. for $q=2$, the qualitative fourth moment theorem was proved in \cite{PT08}. Under different assumptions, this 
result is also proved in \cite{BPsv} where also a qualitative fourth moment theorem for $q=3$ is derived. We would like to mention that the method of proof applied in \cite{BPsv} is rather ad hoc and cannot be generalized to higher values of $q$. We also stress that all existing quantitative fourth moment theorems on the Poisson space make the restrictive assumption that the kernel function $f$ has a constant sign (see e.g. \cite{LRP1, Sch16, ET14}). 
Furthermore, the multiplicative constants in these results depend on the order $q$ in a non-explicit way, { implying that a} statement in the spirit of Corollary \ref{cormt} cannot be inferred from them. We also mention \cite{PZ2}, where one can find a fourth moment theorem for sequences of chaotic elements having the form of homogeneous sums in independent Poisson random variables whose variance is bounded away from zero, as well as \cite{BPjfa}, where the authors prove a fourth moment theorem for multiple integrals with respect of a non-commutative Poisson measure (in the framework of free probability theory), under an additional tameness assumption. { Finally, the already quoted recent contribution \cite{DVZ} also contains several multidimensional extensions of the main findings of the present work.} 
\item { We find it quite remarkable that our fourth moment bounds do not require any additional error term accounting for the discreteness of the Poisson space as is necessary e.g. in the context of degenerate $U$-statistics (see e.g. \cite{deJo90} and \cite{DP16}) as well as for discrete multiple integrals of independent Rademacher random variables (see \cite{DK17}) where such fourth moment theorems { without remainder} do not hold.}
\end{enumerate}
\end{remark}

We also notice the following negative result.

\begin{prop}\label{nogauss}
 For each $q\in\N$, there exists no Gaussian random variable with positive variance in the $q$-th Wiener chaos $C_q$ associated with $\eta$.
\end{prop}

\subsection{Main results on Gamma approximations}

 For $\nu>0$, we denote by $\Gammabar(\nu)$ the so-called \textit{centered Gamma distribution} with parameter $\nu$ which by definition is the distribution of 
\[Z_\nu:=2X_{\nu/2,1}-\nu\,,\]
where, $X_{\nu/2,1}$ has the usual Gamma distribution on $[0,+\infty)$ with shape parameter $\nu/2$ and rate $1$. In particular, one has
\begin{equation*}
 \E[Z_\nu]=0\quad\text{and}\quad \Var(Z_\nu)=\E[Z_\nu^2]=2\nu\,.
\end{equation*}
Moreover, the following moment identity (already exploited in \cite{NP09a}) will play an important role in what follows:
\begin{equation}\label{e:lcm}
 \E[Z_\nu^4]-12\E[Z_\nu^3]-12\nu^2+48\nu=0\,.
\end{equation}
The next result is the counterpart of Theorem \ref{4mt} for centered Gamma approximation.

\begin{theorem}[Fourth moment bound for Gamma approximation]\label{4mtg}
 Fix $\nu>0$ as well as an integer $q\geq1$ and let $F = I_q(f)$ be a multiple Wiener-It\^{o} integral with respect to $\hat{\eta}$, { verifying {\bf Assumption A}}. Assume that $F\in L^4(\Prob)$ and that $\E[F^2]=2\nu$. Also, let $Z_\nu\sim\Gammabar(\nu)$ have the centered Gamma 
 distribution with parameter $\nu$. Then, we have the following bound:
 \begin{align}\label{4mbg}
  d_2(F,Z_\nu)&\leq C_1(\nu)\sqrt{\Babs{\E\bigl[F^4\bigr]-12\E\bigl[F^3\bigr]-12\nu^2+48\nu}}\notag\\
  &\;+C_2(\nu)\biggl(\frac{1}{q}\int_{\mathcal{Z}}\E\bigl[\abs{D_z^+F}^4\bigr]\mu(dz)\biggr)^{1/2}\,,
  \end{align}
where $D^+$ denotes the add-one-cost operator associated with $\eta$ (see Section \ref{framework}) and where we can let 
\begin{align*}
 C_1(\nu)&:=\frac{1}{\sqrt{3}}\max\Bigl(1,\frac{2}{\nu}\Bigr)\quad\text{and}\\
 C_2(\nu)&:=\frac{1}{\sqrt{6}}\max\Bigl(1,\frac{2}{\nu}\Bigr)+\max\Bigr(\sqrt{2\nu},\sqrt{\frac{2}{\nu}}+\sqrt{\frac{\nu}{2}}\Bigr)\,.
\end{align*}
\end{theorem}

\begin{remark}\label{gammarem}
 \begin{enumerate}[(a)]
  \item { The bound \eqref{4mbg} displays an additional term, not directly connected to moments, that is not present in the estimate \eqref{4mb} for normal approximations. For the time being, it is a challenging open problem to determine whether such a term can be removed.}
  \item { The previous result implies that, if, for $n\in\N$, $F_n \in {\rm Ker}(L+q_nI)$ ($q_n\geq 1$) is a sequence of random variables verifying {\bf Assumption A} and such that
  $$
\lim_{n\to\infty} \E[F_n^2] = 2\nu, \quad \lim_{n\to \infty}\left( \E\bigl[F_n^4\bigr]-12\E\bigl[F_n^3\bigr] \right)=12\nu^2-48\nu,
  $$
  and 
  \begin{equation}\label{e:gz}
  \lim_{n\to\infty}  \int_{\mathcal{Z}}\E\bigl[\abs{D_z^+F_n}^4\bigr]\mu(dz) = 0,
  \end{equation}
  then $F_n$ converges in distribution to $Z_\nu$. { This result largely extends the Gamma limit theorem for double integrals stated in \cite[Proposition 4.7]{PTh}, where \eqref{e:gz} is replaced by the requirement that $f_n\to 0$ in $L^4$, with $f_n$ denoting the function in $q_n$ variables such that $F_n = I_{q_n}(f_n)$. }.  In general, if $q_n \equiv q$, and $F_n$ has the form $I_q(f_n)$ for some sufficiently regular kernel $f_n$, then one sufficient condition in order to have \eqref{e:gz} is that all contractions of the type $f_n\star_b^a f_n$ { with $a<b$} converge to zero in $L^2$, where the definition of $f_n\star_b^a f_n$  can be found e.g. in \cite[Section 6]{Lastsv}; see the computations contained in \cite[p. 465-466]{PSTU}. A detailed discussion of the Gamma bound \eqref{4mbg} via the use of contraction operators (in the sprit e.g. of \cite{PTh, FT}) seems to be outside the scope of the present work, and will be tackled elsewhere; see also \cite{DP16b}.
 }
\item { The parametrization of the Gamma distribution which we use is mainly chosen for historical reasons (
see e.g. the papers \cite{NP09a, PTh}) and for convenience because multiple integrals are naturally centered. Of course, 
in general, the Gamma distribution also involves an unimportant rate parameter $\lambda>0$. This situation can be easily dealt with by considering $\lambda^{-1} Z_\nu$ in place of $Z_\nu$ and by using the simple inequality 
$d_2(F, \lambda^{-1} Z_\nu)\leq \lambda^{-1}d_2(\lambda F, Z_\nu)$ which can be bounded by means of 
\eqref{4mbg} whenever $F$ is in $C_q$ and $\E[F^2]=2\nu\lambda^{-2}$.}
 \end{enumerate}
\end{remark}

\subsection{Plan} The paper is organised as follows. Section 2 contains preliminary results concerning stochastic analysis on the Poisson space. Section 3 focusses on several new estimates for multiple integrals, whereas Section 4 and Section 5 deal with the proofs of our main results. Finally, Section 6 contains the proofs of some technical lemmas.

\section*{Acknowledgments} The authors would like to thank the two anonymous referees, the Associate Editor and Guangqu Zheng for insightful remarks and comments. We are also grateful to Solesne Bourguin, Simon Campese, G\"unter Last and Matthias Schulte for useful discussions. Support is acknowledged from the grant {\tt F1RMTH-PUL-15STAR}
(STARS) at Luxembourg University.

\smallskip

\section{Elements of stochastic analysis on the Poisson space}\label{framework}

In this section, we describe our theoretical framework in more detail, {{} by adopting the language of \cite{Lastsv}, corresponding to Chapter 1 in \cite{PecRei16}}. See also the monograph \cite{LPbook}.

\subsection{Setup}\label{ss:setup}  In what follows, we will view the Poisson process $\eta$ as a random element {{} taking values} in the space $\mathbf{N}_\sigma=\mathbf{N}_\sigma(\mathcal{Z})$ of all $\sigma$-finite point measures $\chi$ on $(\mathcal{Z},\mathscr{Z})$ that satisfy $\chi(B)\in\N_0\cup\{+\infty\}$ for all $B\in\mathscr{Z}$. {{} Such a space is} equipped with the smallest 
$\sigma$-field $\mathscr{N}_\sigma:=\mathscr{N}_\sigma(\calZ)$ {{} such that}, for each $B\in\scrZ$,  the mapping $\mathbf{N}_\sigma\ni\chi\mapsto\chi(B)\in[0,+\infty]$ is measurable. {{} As anticipated, throughout the paper we shall assume that the process $\eta$ is {\it proper}, in the sense that $\eta$ can be $\Prob$-a.s. represented in the form
$$
\eta = \sum_{n=1}^{\eta(\mathcal{Z})}\delta_{X_n},
$$
where $\{X_n : n\geq 1\}$ denotes a countable collection of random elements with values in $\mathcal{Z}$ and where, for $z\in \mathcal{Z}$, we write $\delta_z$ for the \textit{Dirac measure} at $z$. A sufficient condition for $\eta$ to be proper is e.g. that $(\mathcal{Z},\mathscr{Z})$ is a Polish space endowed with its Borel $\sigma$-field, with $\mu$ being $\sigma$-finite as above. See \cite[Section 6.1]{LPbook} and \cite[p. 2-3]{Lastsv} for more details. {{} Furthermore, Corollary 3.7 in \cite{LPbook} states that for each Poisson process $\eta$, there exists (maybe on a different probability space) a proper Poisson process $\eta^*$ which has the same distribution as $\eta$. Since all our results depend uniquely on the distribution of $\eta$, it is no restriction of generality to assume that $\eta$ is proper.} 

\smallskip

Now denote by $\mathbf{F}(\mathbf{N}_\sigma)$ the class of all measurable functions $\mathfrak{f}:\mathbf{N}_\sigma\rightarrow\R$ and by $\mathcal{L}^0(\Omega):=\mathcal{L}^0(\Omega,\F)$ the class of real-valued, measurable functions $F$ on $\Omega$. 
Note that, as $\F=\sigma(\eta)$, each $F\in \mathcal{L}^0(\Omega)$ can be written as $F=\mfk(\eta)$ for some measurable function $\mfk$. This $\mfk$, called a {\it representative} of $F$, is $\Prob_\eta$-a.s. uniquely defined, where $\Prob_\eta=\Prob\circ\eta^{-1}$ is the image measure of $\Prob$ under $\eta$ on the space $(\mathbf{N}_\sigma,\mathscr{N}_\sigma)$. 

\smallskip

Using a representative $\mfk$ of $F$, we can define the so-called \textit{add-one cost operator} $D^+=(D_z^+)_{z\in\mathcal{Z}}$ on $\mathcal{L}^0(\Omega)$ (recall that we assume $\F=\sigma(\eta)$) by 
\begin{equation}\label{defdp}
 D_z^+F:=\mfk(\eta+\delta_z)-\mfk(\eta)\,,\quad z\in\mathcal{Z};
\end{equation}
similarly, we define $D^-$ on $\mathcal{L}^0(\Omega)$ via
\begin{equation}\label{defdm}
 D_z^-F:=\mfk(\eta)-\mfk(\eta-\delta_z)\,,\,\, \mbox{ if\ \ } z\in {\rm supp}({\eta}) \,,\,\, \mbox{and\ \ } D_z^-F:=0, \,\, \mbox{otherwise,} \end{equation}
{ where, here, 
\[{\rm supp}({\chi}):=\bigl\{z\in\mathcal{Z}\,:\, \text{for all $A\in\mathscr{Z}$ s.t. $z\in A$: $\chi(A)\geq1$}\bigr\}\]
stands for the support of the measure $\chi\in{\bf N}_\sigma$}; note that, since $\eta$ is proper, if $z\in {\rm supp}(\eta)$, then $\eta-\delta_z \in {\bf N}_\sigma$. Intuitively, $-D^-$ is a \textit{remove-one cost operator}. We stress that the definitions of $D^+F$ and $D^-F$ are, respectively, $\Prob\otimes \mu$-a.e. and $\Prob$-a.s. independent of the choice of the representative $\mfk$ --- {{} see e.g. \cite[Lemma 2.4]{LaPen11} for the case of $D^+$, whereas the case of $D^-$ can be dealt with by using the Mecke formula \eqref{mecke2} below}. {{} Similarly, the conditions stated in {\bf Assumption ${\bf A}$} and  {\bf Assumption ${\bf A}^{\rm (loc)}$} do not depend on the choice of the representative $\mfk$.}

\smallskip

We conclude the section by observing that the operator $D^+$ can be canonically iterated by setting $D^{(1)}:=D^{+}$ and, for $n\geq2$ and $z_1,\dotsc,z_n\in\calZ$ and $F\in\mathcal{L}^0(\Omega)$, by recursively defining
\begin{equation*}
 D^{(n)}_{z_1,\dotsc,z_n}F:=D^{+}_{z_1}\bigl(D^{(n-1)}_{z_2,\dotsc,z_n}F\bigr).
\end{equation*}

\smallskip

\subsection{$L^1$ theory: Mecke formula and $\Gamma_0$} A central formula in the theory of Poisson processes is the so-called \textit{Mecke formula} from \cite{Mecke} which says that for each measurable function $h:\mathbf{N}_\sigma\times\calZ\rightarrow [0,+\infty]$ the identity
\begin{equation}\label{mecke}
\E\biggl[\int_\calZ h(\eta+\delta_z,z)\mu(dz)\biggr] = \E\biggl[\int_\calZ h(\eta,z)\eta(dz)\biggr]
\end{equation}
holds true; see \cite[Chapter 4]{LPbook} for a modern discussion of this fundamental result. {{} We will pervasively use the following consequence of \eqref{mecke}:
\begin{lemma}\label{l:mecke} For some integer $d\geq 1$, let $\mfk_1,...,\mfk_d$ be measurable mappings from $\mathbf{N}_\sigma$ into $[0,+\infty]$, and let $V : [0,+\infty]^{2d}\to [0,+\infty]$ be measurable. Then,
\begin{eqnarray}\notag
&&\E \left[\int_\calZ V(z)\mu(dz)\right]:= \E\left[\int_\calZ V\big(\mfk_1(\eta), \mfk_1(\eta+\delta_z),...,\mfk_d(\eta), \mfk_d(\eta+\delta_z)\big) \mu(dz)\right] \\
&& = \E\left[\int_\calZ V\big(\mfk_1(\eta-\delta_z), \mfk_1(\eta),...,\mfk_d(\eta-\delta_z), \mfk_d(\eta)\big)  \eta(dz)\right].\label{mecke2}
\end{eqnarray}
Both sides of \eqref{mecke2} do not change if any of the $\mfk_i$, $i=1,...,d$ is replaced with another measurable mapping $\widehat{\mfk}_i$ such that $\mfk_i = \widehat{\mfk}_i$, a.s.-$\Prob_\eta$.
\end{lemma}
\begin{proof} Apply relation \eqref{mecke} to the random function
\begin{eqnarray*}
&&h(\eta+\delta_z, z) :=V(z) = V\big(\mfk_1(\eta), \mfk_1(\eta+\delta_z),...,\mfk_d(\eta), \mfk_d(\eta+\delta_z)\big) \\
&&= V\big(\mfk_1(\eta+\delta_z-\delta_z), \mfk_1(\eta+\delta_z),...,\mfk_d(\eta+\delta_z-\delta_z), \mfk_d(\eta+\delta_z)\big){\bf 1}_{\{(\eta+\delta_z)(\{z\})\geq 1 \} },
\end{eqnarray*}
in such a way that
$$
h(\eta, z) = V\big(\mfk_1(\eta-\delta_z), \mfk_1(\eta),...,\mfk_d(\eta-\delta_z), \mfk_d(\eta)\big){\bf 1}_{\{\eta(\{z\})\geq 1 \} }.
$$
The last sentence in the statement follows from \cite[Lemma 2.4]{LaPen11}.
\end{proof}
\begin{remark}\label{r:ovvio}
Plainly, formulae \eqref{mecke} and \eqref{mecke2} continue to hold when the functions $h(\eta+\delta_z, z)$ and $V\big(z\big)$ are in $L^1 (\Prob\otimes \mu)$, without necessarily having a constant sign. 
\end{remark}
}

\smallskip

{{} For random variables $F,G\in \mathcal{L}^0(\Omega)$ such that $D^+F\, D^+G\in L^1(\Prob\otimes \mu)$, we define
\begin{equation}\label{e:Gamma0}
\Gamma_0(F,G) := \frac12\left\{  \int_\calZ (D_z^+F D_z^+G) \, \mu(dz)  + \int_\calZ (D_z^-F D_z^-G) \, \eta(dz)\right\}
\end{equation} 
which verifies $\E[|\Gamma_0(F,G) |] <\infty$, and $\E[\Gamma_0(F,G)] =\E[  \int_\calZ (D_z^+F D_z^+G) \,  \mu(dz) ]$, in view of the Mecke formula \eqref{mecke2}. The following statement will play a fundamental role in our work.

\begin{lemma}[$L^1$ integration by parts]\label{l:ibpl1} Let $G,H \in \mathcal{L}^0(\Omega)$ be such that $$G D^+H, \,\, D^+G\, D^+H\in L^1(\Prob\otimes \mu).$$  
Then,
\begin{equation}\label{e:ibpl1}
\E\left[ G\left(\int_\calZ D_z^+H\, \mu(dz) - \int_\calZ D_z^ -H\, \eta(dz)\right)\right] = -\E[\Gamma_0(G,H)]. 
\end{equation}
\end{lemma}
\begin{proof}
The assumptions in the statement imply that $(G + D^+ G)D^+ H\in L^1(\Prob\otimes \mu)$. Applying \eqref{mecke2} and Remark \ref{r:ovvio} to 
$$
V(z) = \mgk(\eta+\delta_z)\{ \mhk(\eta+\delta_z) - \mhk(\eta)\},
$$
where $\mgk$ and $\mhk$ are representatives of $G$ and $H$, respectively, yields that
$$
\E\left[ G \int_\calZ D_z^ -H\, \eta(dz)\right] = \E\left[ \int_\calZ (G+D^+_zG) D_z^ +H\, \mu(dz)\right], 
$$  
which gives immediately the desired conclusion.
\end{proof}

}

\smallskip

\subsection{$L^2$ theory, part 1: multiple integrals}\label{L21}
 For an integer $p\geq1$ we denote by $L^2(\mu^p)$ the Hilbert space of all square-integrable and real-valued functions on $\mathcal{Z}^p$ and we write $L^2_s(\mu^p)$ 
for the subspace of those functions in $L^2(\mu^p)$ which are $\mu^p$-a.e. symmetric. Moreover, for ease of notation, we denote by $\Enorm{\cdot}$ and $\langle \cdot,\cdot\rangle_2$ the usual norm and scalar product 
on $L^2(\mu^p)$ for whatever value of $p$. We further define $L^2(\mu^0):=\R$. For $f\in L^2(\mu^p)$, we denote by $I_p(f)$ the \textit{multiple Wiener-It\^o integral} of $f$ with respect to $\hat{\eta}$. 
If $p=0$, then, by convention, 
$I_0(c):=c$ for each $c\in\R$.
We refer to Section 3 of \cite{Lastsv} for a precise definition and the following basic properties of these integrals in the general framework of a $\sigma$-finite measure space $(\mathcal{Z},\mathscr{Z},\mu)$. Let $p,q\geq0$ be integers: 
\begin{enumerate}[1)]
 \item $I_p(f)=I_p(\tilde{f})$, where $\tilde{f}$ denotes the \textit{canonical symmetrization} of $f\in L^2(\mu^p)$, i.e. with $\mathbb{S}_p$ the symmetric group acting on $\{1,\dotsc,p\}$ we have
 \[\tilde{f}(z_1,\dotsc,z_p)=\frac{1}{p!}\sum_{\pi\in\mathbb{S}_p} f(z_{\pi(1)},\dotsc,z_{\pi(p)})\,.\]
 \item $I_p(f)\in L^2(\Prob)$, and $\E\bigl[I_p(f)I_q(g)\bigr]= \delta_{p,q}\,p!\,\langle \tilde{f},\tilde{g}\rangle_2 $, where $\delta_{p,q}$ denotes \textit{Kronecker's delta symbol}.
\end{enumerate}

\smallskip

For $p\geq0$, the Hilbert space consisting of all random variables $I_p(f)$, $f\in L^2(\mu^p)$, is called the \textit{$p$-th Wiener chaos} associated with $\eta$, and is customarily denoted by $C_p$. It is a crucial fact that every $F\in L^2(\Prob)$ admits 
a unique representation 
\begin{equation}\label{chaosdec}
 F=\E[F]+\sum_{p=1}^\infty I_p(f_p)\,,
\end{equation}
where $f_p\in L_s^2(\mu^p)$, $p\geq1$, are suitable symmetric kernel functions, and the series converges in $L^2(\Prob)$. Identity \eqref{chaosdec} is referred to as the \textit{chaotic decomposition} of the functional $F\in L^2(\Prob)$. 

\smallskip

 From Theorem 2 in \cite{Lastsv} (which is Theorem 1.3 from the article \cite{LaPen11}) it is known that, for all $F\in L^2(\Prob)$ and all $p\geq1$, the kernel $f_p$ in \eqref{chaosdec} is explicitly given by 
\begin{equation}\label{kerform}
 f_p(z_1,\dotsc,z_p)=\frac{1}{p!}\E\bigl[D^{(p)}_{z_1,\dotsc,z_p}F\bigr]\,,\quad z_1,\dotsc,z_p\in\calZ\,.
\end{equation}

The following new lemma, which relies on \eqref{kerform} and whose proof is deferred to Section \ref{proofs}, will be essential for the proof of Theorem \ref{4mt}.

\begin{lemma}\label{chaosorder}
 Let $p,q\geq1$ be integers and let the multiple Wiener-It\^{o} integrals $F=I_p(f)$ and $G=I_q(g)$ be in $L^4(\Prob)$ and given by symmetric kernels $f\in L^2(\mu^p)$ and $g\in L^2(\mu^q)$, respectively. 
 \begin{enumerate}[{\normalfont (a)}]
  \item The product $FG$ has a finite chaotic decomposition of the form\\
  $FG=\sum_{r=0}^{p+q}\proj{FG}{r}=\sum_{r=0}^{p+q}I_r(h_r)$ with symmetric kernels $h_r\in L_s^2(\mu^r)$.
  \item The kernel $h_{p+q}$ in {\normalfont (a)} is explicitly given by $h_{p+q}=f\tilde{\otimes} g$, where $f\otimes g\in L^2(\mu^{p+q})$ denotes the tensor product of $f$ and $g$ defined by
  \[f\otimes g(z_1,\dotsc,z_{p+q})=f(z_1,\dotsc,z_p) g(z_{p+1},\dotsc,z_{p+q})\]
  and $f\tilde{\otimes} g$ denotes its canonical symmetrization. 
 \end{enumerate}
\end{lemma}

\begin{remark}\label{corem}
 {{} We stress that the statement of Lemma \ref{chaosorder} {\it is not} a direct consequence of the {product formula} for multiple Wiener-It\^{o} integrals on the Poisson space (see e.g. Proposition 5 in \cite{Lastsv} and the discussion therein),
since such a result assumes the square-integrability of the so-called `star contractions kernels' $f\star_r^l g$ associated with $f$ and $g$. It is easily seen that such an integrability property cannot be directly deduced from the minimal assumptions of Lemma \ref{chaosorder}. }
 \end{remark}

\subsection{$L^2$ theory, part 2: Malliavin operators and carr\'e-du-champ}\label{L22}
 We now briefly discuss Malliavin operators on the Poisson space. 
\begin{enumerate}[\bf (i)]

\item The \textit{domain} $\dom D$ of the {\it Malliavin derivative operator }$D$ is the set of all $F\in L^2(\Prob)$ such that the chaotic decomposition \eqref{chaosdec} of $F$ satisfies $\sum_{p=1}^\infty p\,p!\Enorm{f_p}^2<\infty$. For such an $F$, the random function $\mathcal{Z}\ni z\mapsto D_zF\in L^2(\Prob)$ is defined via
\begin{equation}\label{e:D}
 D_zF=\sum_{p=1}^\infty p I_{p-1}\bigl(f_p(z,\cdot)\bigr)\,,
\end{equation}
{{} whenever $z$ is such that the series is converging in $L^2(\Prob)$ (this happens a.e.-$\mu$), and set to zero otherwise; note that $f_p(z,\cdot)$ is an a.e. symmetric function on $\mathcal{Z}^{p-1}$.} Hence, $DF=(D_zF)_{z\in\mathcal{Z}}$ is indeed an element of $L^2\bigl(\Prob\otimes \mu\bigr)$. 
 It is well-known (see e.g. \cite[Lemma 3.1]{PTh}) that, $F \in \dom D$ if and only if $D^+F\in L^2\bigl(\Prob\otimes \mu\bigr)$, and in this case 
\begin{equation}\label{altDz}
D_zF=D_z^+F,\quad \Prob\otimes \mu{\rm -a.e.} . 
\end{equation}

\item The domain $\dom L$ of the \textit{Ornstein-Uhlenbeck generator} $L$ is the set of those $F\in L^2(\Prob)$ whose chaotic decomposition \eqref{chaosdec} verifies $\sum_{p=1}^\infty p^2\,p!\Enorm{f_p}^2<\infty$ (so that $\dom L \subset \dom D$) and, for $F\in\dom L$, one defines
\begin{equation}\label{defL}
 LF=-\sum_{p=1}^\infty p I_p(f_p)\,.
\end{equation}
By definition, $\E[LF]=0$; also, from \eqref{defL} it is easy to see that $L$ is \textit{symmetric} in the sense that 
\begin{equation*}
 \E\bigl[(LF)G\bigr]=\E\bigl[F(LG)\bigr]
\end{equation*}
for all $F,G\in\dom L$. Note that, from \eqref{defL}, it is immediate that the spectrum of $-L$ is given by the nonnegative integers and that $F\in \dom L$ is an eigenfunction of $-L$ with corresponding eigenvalue $p$ if and only if $F=I_p(f_p)$ for some $f_p\in L^2_s(\mu^p)$, that is: 
\begin{equation*}
 C_p={\rm Ker}(L+pI).
\end{equation*}
For $F\in L^2(\Prob)$ given by \eqref{chaosdec} and $p\in\N_0$ we write 
\begin{equation*}
 \proj{F}{p}=I_p(f_p)
\end{equation*}
for the projection of $F$ onto $C_p$, with $f_0:=\E[F]$.
{{} The following identity, which corresponds to formula (65) in \cite{Lastsv}, will play an important role in the sequel: if $F\in\dom L$ is such that $D^+ F\in L^1(\Prob\otimes\mu)$, then 
\begin{equation}\label{formL}
 LF=\int_ \mathcal{Z} \bigl(D_z^+F\bigr) \mu(dz)-\int_\mathcal{Z}\bigl(D_z^{-}F\bigr)\eta(dz)\,.
\end{equation}

\item In order to deal with bounds in the Kolmogorov distance, we will also exploit the properties of the \textit{Skohorod integral operator} $\delta$ associated with $\eta$, which is characterised by the following {\it duality relation}:
\begin{equation}\label{intparts2}
 \E\bigl[G\delta(u)\bigr]=\E\bigl[\langle DG,u\rangle_{L^2(\mu)}\bigr]\quad\text{for all } G\in\dom D,\, u\in\dom\delta,\,
\end{equation}
where ${\rm dom}\, \delta$ stands for its domain (see \cite[p.14-15]{Lastsv}). Recall that the operator $\delta$ satisfies the classical identity 
\begin{equation}\label{e:ldeltad}
L=-\delta D,
\end{equation}
that has to be understood in the following sense: $F\in {\rm dom } \, L$ if and only if $F\in {\rm dom}\, D$ and $DF\in \dom \delta$, and in this case $\delta D F = -LF$. Also, if $u(\eta, \cdot)\in L^1(\Prob\otimes \mu)\cap {\rm dom}\, \delta $, then
\begin{equation}\label{e:pdelta}
\delta(u) = \int_\Z u(\eta - \delta_z, z) \eta(dz) - \int_\Z u(\eta, z) \mu(dz), \quad \mbox{a.s.-}\Prob;
\end{equation}
see \cite[Theorem 6]{Lastsv} for a proof of this fact.
}

\item As it is customary in the theory of Markov generators, see e.g. \cite{BGL14}, for suitable random variables $F,G\in\dom L$ such that $FG\in\dom L$, we introduce the \textit{carr\'{e}-du-champ operator} $\Gamma$ associated with $L$ by 
\begin{equation}\label{cdc}
 \Gamma(F,G):=\frac{1}{2}\bigl(L(FG)-FLG-GLF\bigr)\,.
\end{equation}
{{} The symmetry of $L$} implies immediately the crucial \textit{integration by parts formula} 
\begin{equation}\label{intparts}
 \E\bigl[(LF)G\bigr]=\E\bigl[F(LG)\bigr]=-\E\bigl[\Gamma(F,G)\bigr].
\end{equation}
The connection between \eqref{intparts} and \eqref{e:ibpl1} will be clarified in the discussion to follow.

\smallskip

\item The domain $\dom L^{-1}$ of the \textit{pseudo-inverse} $L^{-1}$ of $L$ is the class of mean zero elements $F$ of $L^2(\Prob)$. If $F= \sum_{p=1}^\infty I_p(f_p)$ is the chaotic decomposition of $F$, then $L^{-1} $F is given by
\begin{equation*}
 L^{-1} F= - \sum_{p=1}^\infty\frac{1}{p}I_p(f)\,.
\end{equation*}
Note that these definitions imply that $L^{-1}F\in {\rm dom }\, L$ (and therefore $L^{-1}F\in {\rm dom}\, D$), for every $F\in {\rm dom}\, L^{-1}$, and moreover
\begin{eqnarray*}
&& LL^{-1}F=F\quad\text{for all }F\in \dom L^{-1}\quad\text{and}\\ && L^{-1} LF=F-\E[F]\quad\text{for all } F\in\dom L\,.
\end{eqnarray*}
Using the first of these identities as well as \eqref{intparts} we obtain that, {{} for $F,G$ such that $G, \, G\, L^{-1}(F-\E(F)) \in \dom L$,}
\begin{align}\label{cov}
 \Cov(F,G)&=\E\bigl[G\bigl(F-\E[F]\bigr)\bigr]=\E\bigl[G\cdot LL^{-1}\bigl(F-\E[F]\bigr)\bigr]\notag\\
& =-\E\bigl[\Gamma\bigl(G,L^{-1}\bigl(F-\E[F]\bigr)\bigr]\,.
\end{align}
In particular, if $F=I_q(f)$ is a multiple integral of order $q\geq1$ such that $F^2\in \dom\, L$, then $\E[F]=0$, $L^{-1}F= - q^{-1}F$ and 
\begin{align}\label{varIq}
 \Var(F)=\frac{1}{q}\E\bigl[\Gamma(F,F)\bigr]\,.
\end{align}
{{} Note that Lemma \ref{chaosorder} immediately implies that $F^2 = I_q(f)^2 \in \dom L$ if and only if $F\in L^4(\Prob)$. On the other hand, if $G\in {\rm dom}\, D$ and $G D^+(L^{-1} F)$, $D^+(L^{-1} F)\in L^1(\Prob\otimes \mu)$, then combining (in order) \eqref{formL}, \eqref{e:ibpl1} and \eqref{altDz} yields
\begin{equation}\label{e:cov2}
\Cov(F,G)=\E\bigl[G\cdot LL^{-1}\bigl(F-\E[F]\bigr)\bigr] =- \E[\Gamma_0(G, L^{-1}\bigl(F-\E[F]\bigr))].
\end{equation}

 }

\end{enumerate}

\subsection{Combining $L^1$ and $L^2$ techniques} The following result provides an explicit representation of the carr\'{e}-du-champ operator $\Gamma$ in terms of $\Gamma_0$, as introduced in \eqref{e:Gamma0}. Although such a characterization follows quite straightforwardly from the (classical) results and definitions provided above, we were not able to locate it in the existing literature { (at least not at our level of generality --- see e.g. \cite[Proposition 4.7]{BDbook} for a similar statement in a more restrictive setting)} and we will therefore provide a full proof. It is one of the staples of our approach.

\begin{prop}\label{cdcform}
 For all $F,G\in\dom L$ such that $FG\in\dom L$ and $$DF,\, DG,\, FDG, \,GDF \in L^1(\Prob\otimes\mu),$$ we have that $DF= D^+ F, \, DG= D^+G$, in such a way that $DF\, DG = D^+F\, D^+G  \in L^1(\Prob\otimes\mu)$, and
 \begin{equation}\label{e:identityGamma}
  \Gamma(F,G)=\Gamma_0(F,G),
 \end{equation}
where $\Gamma_0$ is defined in \eqref{e:Gamma0}.
\end{prop}
In order to prove Proposition \ref{cdcform}, we state the following lemma which will be exploited in several occasions.

\begin{lemma}\label{Difflemma}
\begin{enumerate}[{\normalfont (a)}]
 \item For $F\in \mathcal{L}^0(\Omega)$ and $z\in\mathcal{Z}$ we have the identities
 \begin{align}
  D^+_zF^2&=\bigl(D^+_zF\bigr)^2+2F D^+_zF\label{dp2}\\
  D^+_zF^3&=\bigl(D^+_zF\bigr)^3+3F^2 D^+_zF+3F\bigl(D^+_zF\bigr)^2\label{dp3}\\
  D^-_zF^2&=-\bigl(D^-_zF\bigr)^2+2F D_z^- F\label{dm2}\\
  D^-_zF^3&=\bigl(D^-_zF\bigr)^3+3F^2D^-_zF-3F\bigl(D^-_z F\bigr)^2\label{dm3}\,.
\end{align}
\item Let $\psi\in C^1(\R)$ be such that $\psi'$ is Lipschitz with minimum Lipschitz-constant $\fnorm{\psi''}$. Then, for $F\in  \mathcal{L}^0(\Omega)$ and $z\in\mathcal{Z}$, there are random quantities 
$R_\psi^+(F,z)$ and $R_\psi^-(F,z)$ such that 
\begin{equation*}
 \babs{R_\psi^+(F,z)}\leq\frac{\fnorm{\psi''}}{2}\,,\quad \babs{R_\psi^-(F,z)}\leq\frac{\fnorm{\psi''}}{2}
\end{equation*}
and 
\begin{align*}
 D_z^+\psi(F)&=\psi'(F)D_z^+F +R_\psi^+(F,z)\bigl(D^+_zF\bigr)^2\quad\text{and}\\
 D_z^-\psi(F)&=\psi'(F)D_z^-F +R_\psi^-(F,z)\bigl(D^-_zF\bigr)^2\,.
\end{align*}
\end{enumerate}
\end{lemma}

\begin{proof}
The proof of this result is deferred to Section \ref{proofs}. \\
\end{proof}

\begin{remark}\label{Drem}
 Note that, by virtue of \eqref{dp2} and polarization, for $F,G\in \mathcal{L}^0(\Omega)$ and $z\in\calZ$ we also deduce the product rules:
  \begin{eqnarray}\label{e:mix+}
   D^+_z\bigl(FG\bigr)&=& G D^+_z F+F D^+_z G+ \bigl(D^+_zF \bigr)\bigl(D^+_zG\bigr)\\
   D^-_z\bigl(FG\bigr)&=& G D_z^- F+F D_z^- G - \bigl(D^-_zF \bigr)\bigl(D_z^-G\bigr)\label{e:mix-}
  \end{eqnarray}
If, furthermore, $F,G,FG\in\dom D$, then, from \eqref{altDz} we conclude that
\begin{equation}\label{prodform}
 D_z(FG)=GD_zF+FD_zG+(D_zF)(D_zG)\,,\quad z\in\calZ\,,
\end{equation}
for the Malliavin derivative $D$. {{} Relations \eqref{e:mix+}--\eqref{e:mix-} combined with \eqref{e:identityGamma} imply that $\Gamma$ is not a derivation, and confirm the well-known fact that $L$ is not a diffusion operator (see e.g. \cite[Definition 1.11.1]{BGL14} for definitions).}
\end{remark}

\begin{proof}[Proof of Proposition \ref{cdcform}]
We need only prove \eqref{e:identityGamma} --- as the rest of the assertions in the statement follows from elementary considerations. Since our assumptions imply that $D(FG) \in L^1(\Prob\otimes\mu)$, we can apply \eqref{formL} in order to deduce that
\begin{align}
2\Gamma(F,G)=LFG-GLF-FLG &=\int_\calZ D_z^+(FG)\mu(dz)-\int_\calZ D_z^-(FG)\eta(dz)\notag\\
&\;-G \int_\calZ D_z^+F \, \mu(dz) + G\int_\calZ D_z^-F\, \eta(dz) \notag \\
&\;-F \int_\calZ D_z^+G\, \mu(dz) + F\int_\calZ D_z^-G\, \eta(dz).\notag
\end{align}
Using \eqref{e:mix+} and \eqref{e:mix-} yields immediately the desired formula. 
\end{proof}

\section{Identities and estimates for multiple integrals}\label{s:mwi}

We will now prove several important relations involving multiple stochastic integrals of a fixed order $q\geq 1$. They constitute the backbone of the forthcoming proof of Theorem \ref{4mt}. 

\begin{lemma}\label{vargamma}
 Let $q\geq 1$, and consider a random variable $F$ such that $F = I_q(f)\in C_q = {\rm Ker}(L+qI)$ and $\E[F^4]<\infty$. Then, $F, F^2\in {\rm dom}\, L$, and
\begin{eqnarray}
\notag  \Var\bigl(q^{-1}\Gamma(F,F)\bigr)&=&\sum_{p=1}^{2q-1} \Bigl(1-\frac{p}{2q}\Bigr)^2 \Var\bigl(\proj{F^2}{p}\bigr)\\ &\leq& \frac{(2q-1)^2}{4q^2}\bigl(\E\bigl[F^4\bigr]-3 \E[F^2]^2 \bigr)\,. \label{e:cb1}
 \end{eqnarray}
 Moreover, one has also that
 \begin{eqnarray}
 && \label{e:cb2} \frac{1}{q^2} \E[\Gamma(F,F)^2]\leq \E[F^4]\\ 
&& \frac{1}{q} \E[F^2\Gamma(F,F)]\leq \E[F^4] \label{e:cb3}
\end{eqnarray}
\end{lemma}

\begin{proof} From Lemma \ref{chaosorder}, we know that $F^2=I_q(f)^2$ has a chaos decomposition of the form 
\begin{equation}\label{cdFsq}
 F^2=\sum_{p=0}^{2q}\proj{F^2}{p}=\E[F^2]+\sum_{p=1}^{2q-1}\proj{F^2}{p}+ { I_{2q}(g_{2q})}
\end{equation}
with  $g_{2q}=f\tilde{\otimes}f$, thus ensuring that $F^2$ is in the domain of $L$. By homogeneity, without loss of generality we can assume for the rest of the proof that $\E[F^2]=1$. As $LF=-qF$, by the definitions of $\Gamma$ and $L$ we have 
\begin{align}\label{vg1}
 2\Gamma(F,F)&= LF^2-2FLF=\sum_{p=1}^{2q} -p \proj{F^2}{p} +2q \sum_{p=0}^{2q} \proj{F^2}{p}\notag\\
&=\sum_{p=0}^{2q} (2q-p)\proj{F^2}{p}\,.
\end{align}
By orthogonality, one has that
\begin{align*}
 \Var\bigl(q^{-1}\Gamma(F,F)\bigr)&=\frac{1}{4q^2}\sum_{p=1}^{2q}(2q-p)^2\Var\bigl(\proj{F^2}{p}\bigr)\\
 &=\frac{1}{4q^2}\sum_{p=1}^{2q-1}(2q-p)^2\Var\bigl(\proj{F^2}{p}\bigr),
\end{align*}
proving the first equality in \eqref{e:cb1}. For the inequality, first note that from \eqref{cdFsq} and the isometry property of multiple integrals we have 
\begin{align}\label{vg2}
 \E\bigl[F^4\bigr]-1&=\Var\bigl(F^2\bigr)=\sum_{p=1}^{2q}\Var\bigl(\proj{F^2}{p}\bigr)\notag\\
 &=\sum_{p=1}^{2q-1}\Var\bigl(\proj{F^2}{p}\bigr) +(2q)!\Enorm{f\tilde{\otimes} f}^2\,.
\end{align}
Now, identity (5.2.12) in the book \cite{NouPecbook} yields that 
\begin{equation}\label{vg3}
 (2q!)\Enorm{f\tilde{\otimes} f}^2=2(q!)^2\Enorm{f}^4 +D_q\,,
\end{equation}
where $D_q\geq0$ is a finite non-negative quantity that can be expressed in terms of the contraction kernels associated with $F$, and whose explicit form is immaterial for the present proof. Also, 
\begin{equation*}
 2(q!)^2\Enorm{f}^4=2\Bigl(\E\bigl[F^2\bigr]\Bigr)^2=2,
\end{equation*}
and we deduce from \eqref{vg2} and \eqref{vg3} that 
\begin{align*}
 \frac{(2q-1)^2}{4q^2}\Bigl(\E\bigl[F^4\bigr]-3\Bigr)&=\frac{(2q-1)^2}{4q^2}\sum_{p=1}^{2q-1}\Var\bigl(\proj{F^2}{p}\bigr) +\frac{(2q-1)^2}{4q^2}D_q\\
 &\geq \frac{(2q-1)^2}{4q^2}\sum_{p=1}^{2q-1}\Var\bigl(\proj{F^2}{p}\bigr)\\
 &\geq \frac{1}{4q^2}\sum_{p=1}^{2q-1}(2q-p)^2\Var\bigl(\proj{F^2}{p}\bigr)\\
 &=\Var\bigl(q^{-1}\Gamma(F,F)\bigr)\,,
\end{align*}
which is exactly the second estimate in \eqref{e:cb1}. Relations \eqref{e:cb2} and \eqref{e:cb3} are immediate consequences of \eqref{cdFsq} and \eqref{vg1}.
\end{proof}

The following result will allow us to effectively control residual quantities arising from the applicaton of Stein's method on the Poisson space.

\begin{lemma}\label{remlemma}
Let $q\geq 1$ be an integer and let $F\in L^4(\Prob)$ be an element of the $q$-th Wiener chaos $C_q$, such that $F$ verifies {\bf Assumption A}. Then,
\begin{equation*}
\frac{1}{2q}\int_\calZ\E\bigl[\abs{D_z^+F}^4\bigr]\mu(dz)=\frac{3}{q}\E\bigl[F^2\Gamma(F,F)\bigr]-\E\bigl[F^4\bigr]\leq \frac{4q-3}{2q}\Bigl(\E\bigl[F^4\bigr]-3\E[F^2]^2\Bigr).
\end{equation*}
\end{lemma}

\begin{proof} Again by homogeneity, we can assume without loss of generality that $F$ has unit variance. Observe that $F\in {\rm dom}\, D$, and therefore $DF = D^+F$ (up to a $\Prob\otimes\mu$-negligible set), and also, by virtue of Proposition \ref{cdcform}, one has that $\Gamma(F,F) = \Gamma_0(F,F)$, a.s.-$\Prob$. It follows that
$$
\E\left[ F^2 \int_\calZ (D_z^+F)^2 \mu(dz)\right] \leq 2\E\left[ F^2\Gamma_0(F,F) \right]=2\E\left[ F^2\Gamma(F,F) \right] \leq 2q \E[F^4]<\infty,
$$
where we have used \eqref{e:cb3} { as well as the fact that the integral of a non-negative {function} with respect to the 
non-compensated Poisson measure $\eta$ is non-negative.} Moreover, by Cauchy-Schwarz,
$$
\E\biggl[ |F| \int_\calZ |D_z^+F|^3 \mu(dz)\biggr] \leq \E\biggl[ F^2 \int_\calZ (D_z^+F)^2 \mu(dz)\biggr]^{1/2}
\E\biggl[ \int_\calZ (D_z^+F)^4 \mu(dz)\biggr]^{1/2}<\infty,
$$
so that $F^2 (D^+F)^2, \, F (D^+ F)^3\in L^1(\Prob\otimes \mu)$. Since $LF = -qF$ and $DF\in L^1(\Prob\otimes \mu)$, one infers from \eqref{formL} that
$$
F = -\frac1q \left( \int_\calZ (D_z^+F) \mu(dz) -  \int_\calZ (D_z^-F)\eta(dz) \right).
$$
Since the above discussion also implies that $F^3D^+F, \, D^+(F^3) D^+ F \in L^1(\Prob\otimes \mu)$ (via \eqref{dp3}), we can now exploit the integration by parts relation stated in Lemma \ref{l:ibpl1} to deduce that
\begin{equation*}
\E\bigl[F^4\bigr]=-\frac1q\E\left[F^3  \left( \int_\calZ (D_z^+F) \mu(dz) -  \int_\calZ (D_z^-F)\eta(dz) \right) \right]=\frac1q \E\bigl[\Gamma_0(F,F^3)\bigr].
\end{equation*}
Now, using \eqref{dp3} and \eqref{dm3} we obtain
\begin{align*}
\Gamma_0(F,F^3)&=\frac{1}{2}\biggl(\int_\calZ D^+_zF \Bigl((D^+_zF)^3+3F^2D^+_zF +3F(D^+_zF)^2\Bigl)\mu(dz)\\
&\;+ \int_\calZ D^-_zF\Bigl((D^-_zF)^3+3F^2D^-_zF -3F(D^-_zF)^2\Bigl)\eta(dz)\biggr)\\
&=\frac{1}{2}\biggl(\int_\calZ\Bigl((D^+_zF)^4+3F^2(D^+_zF)^2 +3F(D^+_zF)^3\Bigl)\mu(dz)\\
&\;+ \int_\calZ \Bigl((D^-_zF)^4+3F^2(D^-_zF)^2 -3F(D^-_zF)^3\Bigl)\eta(dz)\biggr),
\end{align*}
and we also have 
\begin{align*}
3F^2\Gamma_0(F,F)&= 3F^2\Gamma(F,F)=\frac12\biggl(\int_\calZ 3F^2(D^+_zF)^2\mu(dz)+ \int_\calZ3 F^2(D^-_zF)^2\eta(dz)\biggr) \,.
\end{align*}
Hence, using the Mecke formula \eqref{mecke2} (as well as the content of Remark \ref{r:ovvio}) in the case
$$
V(z) = - \big( \mathfrak{f}(\eta+\delta_z) - \mathfrak{f}(\eta) \big) ^4-3\mathfrak{f}(\eta) \big( \mathfrak{f}(\eta+\delta_z) - \mathfrak{f}(\eta) \big) ^3,
$$
where $\mathfrak{f}$ is some representative of $F$, we can conclude that 
\begin{align*}
&\frac{3}{q}\E\bigl[F^2\Gamma(F,F)\bigr]-\E\bigl[F^4\bigr] =\frac{1}{2q}\E\biggl[\int_\calZ\Bigl(-(D^+_zF)^4-3F(D^+_zF)^3\Bigr)\mu(dz)\\
&\hspace{3cm}+\int_\calZ\Bigl(-(D^-_zF)^4+3F(D^-_zF)^3\Bigr)\eta(dz)\biggr]\\
&=\frac{1}{2q}\E\biggl[-2\int_\calZ (D^+_zF)^4\mu(dz)+3\int_\calZ (D^+_zF)^3\bigl(\mathfrak{f}(\eta+\delta_z)-\mathfrak{f}(\eta)\bigr)\mu(dz)\biggr]\\
&=\frac{1}{2q}\E\biggl[\int_\calZ (D^+_zF)^4\mu(dz)\biggr]\,.
\end{align*}
Finally, using relations \eqref{cdFsq} and \eqref{vg1} from the proof of Lemma \ref{vargamma}, we obtain
\begin{align}\label{pmt1}
&\frac{1}{q}\int_\calZ\E\bigl[\abs{D_z^+F}^4\bigr]\mu(dz) =2\biggl(\frac{3}{q}\E\bigl[F^2\Gamma(F,F)\bigr]-\E\bigl[F^4\bigr]\biggr)\notag\\
&=2\biggl(\frac{3}{q}q\bigl(\E[F^2]\bigr)^2-\E\bigl[F^4\bigr]+\frac{3}{2q}\sum_{p=1}^{2q-1}(2q-p)\Var\bigl(\proj{F^2}{p}\bigr)\biggr)\notag\\
&\leq 2\Bigl(3-\E\bigl[F^4\bigr]\Bigr)+\frac{3(2q-1)}{q}\sum_{p=1}^{2q-1}\Var\bigl(\proj{F^2}{p}\bigr)\notag\\
&\leq 2\Bigl(3-\E\bigl[F^4\bigr]\Bigr)+\frac{3(2q-1)}{q}\Bigl(\E\bigl[F^4\bigr]-3\Bigr)\notag\\
&=\frac{4q-3}{q}\Bigl(\E\bigl[F^4\bigr]-3\Bigr)\,,
\end{align}
where the last inequality is again a consequence of \eqref{vg2} and \eqref{vg3}.\end{proof}

We eventually prove an estimate that will be crucial in order to deal with bounds in the Kolmogorov distance.

\begin{lemma}\label{l:indicator} For some fixed $q\geq 1$, let $F\in {\rm Ker}(L+qI)$ satisfy both {\bf Assumption A} and {\bf Assumption ${\bf A}^{\rm (loc)}$}. Then, 
\begin{equation}\label{e:indicator}
0\leq \frac1q  \sup_{x\in \R} \E\left[ \int_\calZ (D^+_z {\bf 1}_{\{ F>x\}} |D_z^+F| D_z^+ F \,  \mu(dz)\right] \leq 10\sqrt{\E[F^4] - 3\E[F^2]^2} .
\end{equation}
\end{lemma}
\begin{proof}
One checks immediately that $D^+_z {\bf 1}_{\{ F>x\}} D_z^+F \geq 0$, so that we need only prove the second inequality in the statement; also, without loss of generality and by homogeneity, we can once again assume that $F$ has unit variance. According to \eqref{e:D}--\eqref{altDz}, we can choose a version of $D^+F$ such that, for $\mu$--almost every $z\in \calZ$, the random variable $D_z^+F = D_zF$ is an element of the $(q-1)$th Wiener chaos $C_{q-1}$. Applying Lemma \ref{vargamma} and Lemma \ref{remlemma} to every $D_zF$ such that $z$ lies outside the exceptional set, one therefore infers that
\begin{eqnarray}\label{e:bozo1}
 A&:=& \int_\calZ \E\left[ \int_\calZ \bigl(D_{z_2} D_{z_1} F\bigr)^4 \mu(dz_2)\right] \mu(dz_1)\notag\\
 &=&  \int_\calZ \E\left[ \int_\calZ \bigl(D^+_{z_2} (D^+_{z_1} F)\bigr)^4 \mu(dz_2)\right] \mu(dz_1) \notag\\
&\leq&  4(q-1) \E\left[\int_{\calZ} (D_z^+F)^4 \mu(dz)\right] \leq 16q(q-1)\Bigl(\E\bigl[F^4\bigr]-3 \Bigr),
\end{eqnarray}
{ where we have applied twice (first to $D_{z_1}^+F$ and then to $F$ itself) the following inequality which is immediate from Lemma \ref{remlemma}:
\begin{equation}\label{e:bozo3}
C:= \int_\calZ\E\bigl[\abs{D_z^+F}^4\bigr]\mu(dz)\leq 4q \Bigl(\E\bigl[F^4\bigr]-3 \Bigr).
\end{equation}

Moreover, using the definition \eqref{e:Gamma0} of $\Gamma_0$, as well as \eqref{e:cb3} and again \eqref{e:bozo3} we obtain}
\begin{eqnarray}B&:=& \int_\calZ \E\left[ (D_{z_1} F)^2 \int_\calZ (D_{z_2} D_{z_1} F)^2 \mu(dz_2)\right] \mu(dz_1) \notag \\  &=& \int_\calZ \E\left[ (D^+_{z_1} F)^2 \int_\calZ (D^+_{z_2} D^+_{z_1} F)^2 \mu(dz_2)\right] \mu(dz_1) \notag\\
& \leq & 2(q-1) \int_\calZ \E\left[ (D^+_{z_1} F)^2 \Gamma_0(D^+_{z_1} F , D^+_{z_1} F) \right] \mu(dz_1) \notag \\
&=& 2(q-1) \int_\calZ \E\left[ (D^+_{z_1} F)^2 \Gamma(D^+_{z_1} F , D^+_{z_1} F) \right] \mu(dz_1) \notag \\
&\leq &2(q-1) \E\left[\int_\calZ (D^+_{z_1} F)^4\mu(dz_1)\right] \leq 8q(q-1)\Bigl(\E\bigl[F^4\bigr]-3 \Bigr)\,.
   \label{e:bozo2}
\end{eqnarray}
Now write $\Phi(a) := a|a|$, $a\in \R$. In view of the inequality (proved e.g. in \cite[Section 4.2]{PTh})
\begin{equation}\label{e:pth}
[D^+_{z_2}\Phi(D^+_{z_1}F)]^2 \leq 8(D^+_{z_1}F)^2(D^+_{z_2} D^+_{z_1} F)^2+ 2(D^+_{z_2} D^+_{z_1} F)^4,
\end{equation}
valid $\mu^2$--almost everywhere, we deduce immediately that the process $z\mapsto v(z) := \Phi(D^+_z F)$ is such that $v(z)\in {\rm dom}\, D$ for $\mu$-almost every $z$, and $v\in {\rm dom}\, \delta$ --- this last fact being a consequence of the classical criterion stated in \cite[Theorem 5]{Lastsv} and of the estimates \eqref{e:bozo1}--\eqref{e:bozo2}, together with the fact that $\E[F^4]<\infty$ by assumption. Also, in view of the fact that $v\in L^1(\Prob\otimes \mu)$ by assumption, equation \eqref{e:pdelta} yields that 
$$
\delta(v) = \int_\calZ \Phi(D^-_z F) \eta(dz) - \int_\calZ \Phi(D^+_z F) \mu(dz). 
$$
We now fix $x\in \R$. Relation \eqref{mecke2} applied to the mapping
$$
V(z) = {\bf 1}_{\{ \mathfrak{f}(\eta+\delta_z)>x\}} \Phi \big( \mathfrak{f}(\eta+\delta_z) - \mathfrak{f}(\eta) \big),
$$
where $\mathfrak{f}$ is a representative of $F$, yields that 
\begin{eqnarray*}
&&\frac1q \E\left[ \int_\calZ D^+_z {\bf 1}_{\{ F>x\}} |D_z^+F| D_z^+ F \,  \mu(dz)\right]\\ &&= \frac1q \E\left[{\bf 1}_{\{ F>x\}} \left(  \int_\calZ \Phi(D^-_z F) \eta(dz) - \int_\calZ \Phi(D^+_z F) \mu(dz)\right)\right]\\ && = \frac1q\E\left[{\bf 1}_{\{ F>x\}}\delta(v)\right ] \leq  \frac1q \E\left[ \delta(v)^2\right]^{1/2}.
\end{eqnarray*}
To conclude, we use \cite[formula (56)]{Lastsv} as well as \eqref{e:pth} to deduce that
\begin{align*}
 \E\left[ \delta(v)^2\right] &\leq \E\left[ \int_\calZ v(z)^2\mu(dz)\right] + \E\left[ \int_\calZ \int_\calZ (D^+_y v(z))^2\mu(dz)\mu(dy)\right]\notag\\
  &\leq C+8B+2A\leq\bigl(4q+64q(q-1)+32q(q-1)\bigr)\Bigl(\E\bigl[F^4\bigr]-3 \Bigr)\notag\\
 &\leq100 q^2\Bigl(\E\bigl[F^4\bigr]-3 \Bigr)\,,
\end{align*}

which in turn implies that 
$$
\frac1q \E\left[ \delta(v)^2\right]^{1/2}\leq 10 \sqrt{\E\bigl[F^4\bigr]-3}\,,  
$$
where $A, B,C$ have been defined above, and where we have used the estimates \eqref{e:bozo1}--\eqref{e:bozo3}.
\end{proof}

\section{Proof of Theorem \ref{4mt}}\label{proof4mt}

In order to prove Theorem \ref{4mt} we have to establish new abstract bounds on the normal approximation of functionals on the Poisson space in the Wasserstein and Kolmogorov distances, respectively. Recall the definition of $\Gamma_0$ given in \eqref{e:Gamma0}.

\begin{prop}\label{genbound}
 Let $F\in \dom D$ be such that $\E[F]=0$ and let $N\sim \mathscr{N}(0,1)$ be a standard normal random variable. Assume that
\begin{equation}\label{e:assgb1}
D^+(L^{-1}F), \, FD^+(L^{-1} F) \in L^1(\Prob\otimes \mu).
\end{equation} 
 Then, we have the bounds
 \begin{align}
  d_1(F,N)&\leq \sqrt{\frac{2}{\pi}}\E\Babs{1-\Gamma_0\bigl(F,-L^{-1}F\bigr)} +\int_{\mathcal{Z}}\E\Bigl[\babs{D_z^+F}^2\babs{D_z^+L^{-1}F}\Bigr] \mu(dz)   \label{gb1}\\
  &\leq\sqrt{\frac{2}{\pi}}\babs{1-\E[F^2]}+ \sqrt{\frac{2}{\pi}}\sqrt{\Var\bigl(\Gamma_0(F,-L^{-1}F)\bigr)}\notag\\
  &\; +\int_{\mathcal{Z}}\E\Bigl[\babs{D_z^+F}^2\babs{D_z^+L^{-1}F}\Bigr]\mu(dz)\,. \label{gb2}
 \end{align}
If, furthermore, $F=I_q(f)$ for some $q\geq 1$ and some square-integrable, symmetric kernel $f$ on $\mathcal{Z}^q$ and $\E[F^2]=q!\Enorm{f}^2=1$, then $-L^{-1}F=q^{-1}F$, 
\begin{align*}
 \E\bigl[\Gamma_0(F,-L^{-1}F)\bigr]&=q^{-1}\E\bigl[\Gamma_0(F,F)\bigr]=1\quad\text{and}\\
 \int_{\mathcal{Z}}\E\bigl[\abs{D_z^+F}^2\abs{D_z^+L^{-1}F}\bigr]\mu(dz)&=q^{-1}\int_{\mathcal{Z}}\E\bigl[\abs{D_z^+F}^3\bigr]\mu(dz)\\
 &\leq\Bigl(q^{-1}\int_{\mathcal{Z}}\E\bigl[\abs{D_z^+F}^4\bigr]\mu(dz)\Bigr)^{1/2}
\end{align*}
so that the previous estimate \eqref{gb2} gives
\begin{align}
  d_1(F,N)&\leq \sqrt{\frac{2}{\pi}}\sqrt{\Var\bigl(q^{-1}\Gamma_0(F,F)\bigr)}+\frac{1}{\sqrt{q}}\Bigl(\int_{\mathcal{Z}}\E\bigl[\abs{D_z^+F}^4\bigr]\mu(dz)\Bigr)^{1/2}\,.\label{sb1}
\end{align}
\end{prop}

\begin{remark} Under the assumptions of Proposition \ref{genbound}, we have that $F, L^{-1}F\in {\rm dom}\, D$, in such a way that $\Gamma_0 (F, -L^{-1} F)$ is an element ot $L^1(\Prob)$. It follows that the variance $\Var\bigl(\Gamma_0(F,-L^{-1}F)\bigr)$ is always well-defined, albeit possibly infinite.

\end{remark}

\begin{proof}[Proof of Proposition \ref{genbound}]
 We apply Stein's method for normal approximation. Define the class $\mathscr{F}_1$ of all continuously differentiable functions $\psi$ on $\R$ such that both $\psi$ and $\psi'$ are Lipschitz-continuous with minimal Lipschitz constants 
\begin{equation}\label{steinsol}
 \fnorm{\psi'}\leq\sqrt{\frac{2}{\pi}}\quad\text{and}\quad\fnorm{\psi''}\leq 2\,.
\end{equation}
 Then, it is well-known (see e.g. Theorem 3 of \cite{BPsv}, and the references therein) that 
 \begin{equation}\label{pgb1}
  d_1(F,N)\leq\sup_{\psi\in\mathscr{F}_1}\babs{\E\bigl[\psi'(F)-F\psi(F)\bigr]}\,.
 \end{equation}
Let us thus fix $\psi\in\mathscr{F}_1$. The Lipschitz property of $\psi$ implies that $\psi(F)\in {\rm dom} D$, whereas the trivial estimate
$$
|\psi(F) D^+(L^{-1} F)| \leq \big(| \psi(0) | +\sqrt{2/\pi}\, |F|\big)\times  |D^+(L^{-1} F)|
$$
implies that $\psi(F) D^+(L^{-1} F)\in L^1(\Prob\otimes \mu)$. Using that $\E[F]=0$ we therefore deduce from \eqref{e:cov2} that 
\begin{align}\label{pgb2}
 \E\bigl[F\psi(F)\bigr]&=\E\bigl[\psi(F)\cdot LL^{-1}F\bigr]=-\E\bigl[\Gamma_0\bigl(\psi(F),L^{-1}F\bigr)\bigr]\,.
\end{align}
Now, by the definition of $\Gamma_0$ and Lemma \ref{Difflemma} (b) we obtain that 
\begin{align}\label{pgb3}
 &2\Gamma_0\bigl(\psi(F),L^{-1}F\bigr)=\int_\calZ\bigl(D_z^+\psi(F)\bigr)\bigl(D_z^+L^{-1}F\bigr)\mu(dz)+\int_\calZ \bigl(D_z^-\psi(F)\bigr)\bigl(D_z^-L^{-1}F\bigr)\eta(dz)\notag\\
&=\psi'(F)\int_\calZ\bigl(D_z^+F\bigr)\bigl(D_z^+L^{-1}F\bigr)\mu(dz)+\int_\calZ R_\psi^+(F,z)\bigl(D_z^+F\bigr)^2\bigl(D_z^+L^{-1}F\bigr)\mu(dz)\notag\\
&\;+\psi'(F)\int_\calZ\bigl(D_z^-F\bigr)\bigl(D_z^-L^{-1}F\bigr)\eta(dz)+\int_\calZ R_\psi^-(F,z)\bigl(D_z^-F\bigr)^2\bigl(D_z^-L^{-1}F\bigr)\eta(dz)\notag\\
&=:\psi'(F)\int_\calZ\bigl(D_z^+F\bigr)\bigl(D_z^+L^{-1}F\bigr)\mu(dz)+R_+ \notag\\
&\;+\psi'(F)\int_\calZ\bigl(D_z^-F\bigr)\bigl(D_z^-L^{-1}F\bigr)\eta(dz) +R_-\notag\\
&=2\psi'(F)\Gamma_0(F,L^{-1}F)+R_+ +R_-
\end{align}
with 
\begin{align}\label{pgb4}
 \E\abs{R_+}&\leq\frac{\fnorm{\psi''}}{2} \E\biggl[\int_\calZ \babs{D_z^+F}^2\babs{D_z^+L^{-1}F}\mu(dz)\biggr]\notag\\
 &\leq \E\biggl[\int_\calZ \babs{D_z^+F}^2\babs{D_z^+L^{-1}F}\mu(dz)\biggr]
\end{align}
and 
\begin{align}\label{pgb5}
 \E\abs{R_-}&\leq\frac{\fnorm{\psi''}}{2} \E\biggl[\int_\calZ \babs{D_z^-F}^2\babs{D_z^-L^{-1}F}\eta(dz)\biggr]\notag\\
 &\leq \E\biggl[\int_\calZ \babs{D_z^-F}^2\babs{D_z^-L^{-1}F}\eta(dz)\biggr]\notag\\
 &=\E\biggl[\int_\calZ \babs{D_z^+F}^2\babs{D_z^+L^{-1}F}\mu(dz)\biggr]\,,
\end{align}
where the last identity holds by virtue of \eqref{mecke2}, as applied to
$$
V(z) = \big(\mathfrak{f}(\eta+\delta_z) - \mathfrak{f}(\eta)\big)^2 \big| \mathfrak{f}^*(\eta+\delta_z) - \mathfrak{f}^*(\eta)\big|,
$$
where $\mathfrak{f}$ is a representative of $F$ and $\mathfrak{f}^*$ is a representative of $L^{-1}F$. Thus, from \eqref{pgb2} and \eqref{pgb3} we infer 
{
\begin{align}\label{pgb6}
 \Babs{\E\bigl[\psi'(F)-F\psi(F)\bigr]}&\leq\Babs{\E\bigl[\psi'(F)\bigl(1-\Gamma_0\bigl(F,-L^{-1}F\bigr)\bigr)\bigr]}+\frac12 \bigl(\E|R_+ |+\E|R_- | \bigr) \,,
\end{align}
}
and from \eqref{steinsol}, \eqref{pgb4}, \eqref{pgb5} and \eqref{pgb6} we conclude that 
\begin{align*}
 \babs{\E\bigl[\psi'(F)-F\psi(F)\bigr]}&\leq\sqrt{\frac{2}{\pi}}\E\babs{1-\Gamma_0\bigl(F,-L^{-1}F\bigr)}+\E\biggl[\int_\calZ \babs{D_z^+F}^2\babs{D_z^+L^{-1}F}\mu(dz)\biggr].
\end{align*}
Plugging such an estimate into \eqref{pgb1} yields \eqref{gb1}. By \eqref{e:cov2} we know that 
\begin{equation*}
 \E\bigl[\Gamma_0(F,-L^{-1}F)\bigr]=\Var(F)=\E[F^2]
\end{equation*}
and, hence, \eqref{gb2} follows from \eqref{gb1} by using the triangle and Cauchy-Schwarz inequalities. To prove \eqref{sb1} we first apply the Cauchy-Schwarz inequality to obtain
\begin{align*}
 \int_{\mathcal{Z}}\E\bigl[\abs{D_z^+F}^3\bigr]\mu(dz)&\leq\Bigl(\int_{\mathcal{Z}}\E\bigl[\abs{D_z^+F}^4\bigr]\mu(dz)\Bigr)^{1/2} \Bigl(\int_{\mathcal{Z}}\E\bigl[\abs{D_z^+F}^2\bigr]\mu(dz)\Bigr)^{1/2}
\end{align*}
But, by using the isometry properties of multiple integrals we have 
\begin{align}\label{iso}
 \int_{\mathcal{Z}}\E\bigl[\abs{D_z^+F}^2\bigr]\mu(dz)&=q^2 \int_{\mathcal{Z}}\E\bigl[I_{q-1}\bigl(f(z,\cdot)\bigr)^2\bigr]\mu(dz)\notag\\
& =q^2 (q-1)!\int_{\mathcal{Z}}\Enorm{f(z,\cdot)}^2\mu(dz)=q q!\Enorm{f}^2=q\E[F^2]=q\,.
\end{align}
Hence, we obtain
\begin{equation*}
 q^{-1}\int_{\mathcal{Z}}\E\bigl[\abs{D_z^+F}^3\bigr]\mu(dz)\leq\frac{1}{\sqrt{q}}\Bigl(\int_{\mathcal{Z}}\E\bigl[\abs{D_z^+F}^4\bigr]\mu(dz)\Bigr)^{1/2} 
\end{equation*}
proving \eqref{sb1}.\\
\end{proof}

The next result provides a similar estimate in the Kolmogorov distance.

\begin{prop}\label{p:genkol} Under the same assumptions as in Proposition \ref{genbound}, one has the bounds
 \begin{eqnarray}
 \dk(F,N)&\leq& \E\Babs{1-\Gamma_0\bigl(F,-L^{-1}F\bigr)}\label{e:k1} \\
 && + \E\left[\big (|F|+\sqrt{2\pi}/4\big) \int_\calZ (D^+_zF)^2|D_z^+L^{-1}F| \mu(dz)\right] \notag \\
 &&\quad\quad\quad +\sup_{x\in \R} \E\left[ \int_\calZ (D_z^+F)\, |D_z^+(L^{-1}F) | \, D_z^+{\bf 1}_{\{ F>x \} } \mu(dz)\right]\notag\\
 &\leq& \babs{1-\E[F^2]}+ \sqrt{\Var\bigl(\Gamma_0(F,-L^{-1}F)\bigr)}\label{e:k2}\\
 &&\quad + \E\left[\left ( \int_\calZ (D^+_z F)^2\mu(dz)\right)^2\right]^{1/4}\bigl(1+\E[F^4]^{1/4}\bigr) \notag \\ 
&&  \quad\quad\times \sqrt{ \E\left[ \int_\calZ (D^+_zF)^2(D^+_z(L^{-1}F) )^2\mu(dz)\right] } \notag \\
 &&\quad\quad\quad\quad+ \sup_{x\in \R} \E\left[ \int_\calZ (D_z^+F)\, |D_z^+(L^{-1}F) |\, D^+{\bf 1}_{\{ F>x\}} \mu(dz)\right]\notag.
 \end{eqnarray}
 If $F=I_q(f)$ for some $q\geq 1$ and some square-integrable, symmetric kernel $f$ on $\mathcal{Z}^q$ and $\E[F^2]=q!\Enorm{f}^2=1$, then \eqref{e:k2} becomes
 \begin{align} \label{e:k3}
 \dk (F,N)& \leq \sqrt{\Var\bigl(q^{-1} \Gamma_0(F,F)\bigr)}\notag\\
 &\;+\frac{1}{q}\bigl(1+\E[F^4]^{1/4}\bigr) \E\left[\left ( \int_\calZ (D^+_z F)^2\mu(dz)\right)^2\right]^{1/4}\sqrt{ \E\left[ \int_\calZ (D^+_zF)^4 \mu(dz)\right] } \notag \\
 &\;+\frac1q\sup_{x\in \R} \E\left[ \int_\calZ (D_z^+F)\,  |D_z^+F | \, D_z^+{\bf 1}_{\{ F>x\}} \mu(dz)\right]\,.
\end{align}

\end{prop}
\begin{proof} Fix $x\in \R$. According to Propositon \ref{p:stein}, we can write
$$
\left| \Prob(F\leq x) - \Prob(N\leq x)  \right| = \left| \E[ g'_x(F) - Fg_x(F)]\right|,
$$
where $g_x$ is the solution of the Stein equation \eqref{eq:stein-equation} associated with $x$, whose properties are stated in Proposition \ref{p:stein}. Using Proposition \ref{p:stein} and reasoning as in the proof of Proposition \ref{genbound}, one deduces that
\begin{eqnarray*}
&&\left| \E[ g'_x(F) - Fg_x(F)]\right|\\ &&\leq  \!\E\Big[| g'_x(F)| |1-\Gamma_0(F, -L^{-1} F)|\Big] \\
&&\quad\quad +\frac14 \E\left[ \left( |F|+ \sqrt{2\pi}/4\right) \int_\calZ (D_z^+F)^2 | D_z^+(L^{-1} F)|\mu(dz) \right] \\
&&\quad\quad + \frac12 \E\left[ \int_\calZ (D^+_z F) | D_z^+(L^{-1} F)|\, D^+_z{\bf 1}_{\{F>x\}} \mu(dz)\right]\\
&&\quad\quad +\frac34 \E\left[ \int_\calZ \left( |F- D_z^{-}F|+ \sqrt{2\pi}/4\right)  (D_z^-F)^2 | D_z^-(L^{-1} F)|\eta(dz) \right]\\
&&\quad\quad + \frac12 \E\left[ \int_\calZ (D^-_z F) | D_z^-(L^{-1} F)|\, D^-_z{\bf 1}_{\{F>x\}} \eta(dz)\right].
\end{eqnarray*}
Note that, in order to obtain the previous estimate, one has to use Point (f) and Point (g) in Proposition \ref{p:stein}, respectively, in order to control the quantities $| D_z^+g_x(F) - g'_x(F)D_z^+F|$ and $|D_z^-g_x(F)- g'_x(F)D^{-}_zF|$. Bound \eqref{e:k1} can now be deduced by applying \eqref{mecke2} to the mappings
$$
V(z) = \left( |\mathfrak{f}(\eta) |+ \sqrt{2\pi}/4\right) \left(\mathfrak{f}(\eta+\delta_z) - \mathfrak{f}(\eta) \right)^2 | \mathfrak{f}^*(\eta+\delta_z) - \mathfrak{f}^*(\eta)|,
$$
and 
$$
V(z) = {\bf 1}_{\{ \mathfrak{f}(\eta+\delta_z)>x\}}  \left(\mathfrak{f}(\eta+\delta_z) - \mathfrak{f}(\eta) \right) | \mathfrak{f}^*(\eta+\delta_z) - \mathfrak{f}^*(\eta)|,
$$
where $\mathfrak{f}$ and $\mathfrak{f}^*$ are representatives of $F$ and $L^{-1}F$, respectively. The estimate \eqref{e:k2} can be deduced by applying  the Cauchy-Schwarz and triangle inequalities to the middle term of \eqref{e:k1}. 
The second part of the statement immediately follows from \eqref{e:k2} and from the fact that, if $F = I_q(f)$, then $-L^{-1}F = q^{-1}F$.
\end{proof}

\begin{proof}[End of the proof of Theorem \ref{4mt}] Since, under {\bf Assumption A}, one has that
$$
\Gamma(F,F) = \Gamma_0(F,F), \quad \mbox{a.s.--}\Prob,
$$
the estimate \eqref{4mb} is a direct consequence of \eqref{sb1}, Lemma \ref{vargamma} and Lemma \ref{remlemma}, as well as of elementary simplifications. Similarly, \eqref{e:4mb3} follows from \eqref{e:k3}, Lemma \ref{vargamma}, Lemma \ref{remlemma} and Lemma \ref{l:indicator}, combined with the estimate
\begin{equation*}
 \E\left[\left ( \int_\calZ (D^+_z F)^2\mu(dz)\right)^2\right]^{1/4}\leq 4^{1/4} \E[(\Gamma_0(F,F))^2]^{1/4}\leq \sqrt{2q}\, \E[F^4]^{1/4},
\end{equation*}
where we have used \eqref{e:cb2}. 
{ Finally, the bound \eqref{e:4mb4} follows from \eqref{e:4mb3} by distinguishing the cases 
\[\E[F^4]>3+\frac{1}{121}\quad\text{and}\quad\E[F^4]\leq 3+\frac{1}{121}\]
and by taking into account the fact that the Kolmogorov distance is bounded by $1$.}
\end{proof}

\begin{proof}[Proof of Proposition \ref{nogauss}] Fix $q\geq 2$. Reasoning as in \cite[Corollary 2]{NP05}, if a Gaussian random variable $F:= I_q(f)\in C_q$ such that $\E[I_q(f)^2]:=c>0$ existed, then $\E[F^4]- 3c^2=0$. Formulae \eqref{vg2}--\eqref{vg3}, together with the explicit form of $D_q$ would therefore imply that $f\otimes_r f=0$ for every $r=1,...,q-1$, where $q$ is the $r$th contraction of $f$ with itself, as defined in \cite[Appendix B]{NouPecbook}. This conclusion contradicts the fact that $c = q! \|f\|_2^2 > 0$. The case $q=1$ follows immediately from the relation $\E[I_1(f)^4] = 3\|f\|^4_2+ \int_\calZ f^4 d\mu$.
\end{proof}


\section{Proof of Theorem \ref{4mtg}}\label{proof4mtg}
We begin by giving the analog of Proposition \ref{genbound} for Gamma approximation.

\begin{prop}\label{genboundg}
 Let $F\in \dom D$ satisfy the same assumptions as in the statement of Proposition \ref{genbound}, and let $Z_\nu\sim \Gammabar(\nu)$ have the centered Gamma distribution with parameter $\nu>0$. Then, we have the bounds
 \begin{align}
  d_2(F,Z_\nu)&\leq\max\Bigl(1,\frac{2}{\nu}\Bigr)\E\Babs{2(F+\nu)-\Gamma_0\bigl(F,-L^{-1}F\bigr)} \notag\\
  &\;+\max\Bigl(1,\frac{1}{\nu}+\frac12\Bigr)\int_{\mathcal{Z}}\E\Bigl[\babs{D_z^+F}^2\babs{D_z^+L^{-1}F}\Bigr] \mu(dz)   \label{gbg1}\\
  &\leq\max\Bigl(1,\frac{2}{\nu}\Bigr)\babs{2\nu-\E[F^2]}+ \max\Bigl(1,\frac{2}{\nu}\Bigr)\sqrt{\Var\Bigl(2F-\Gamma_0\bigl(F,-L^{-1}F\bigr)\Bigr)}\notag\\
  &\; +\max\Bigl(1,\frac{1}{\nu}+\frac12\Bigr)\int_{\mathcal{Z}}\E\Bigl[\babs{D_z^+F}^2\babs{D_z^+L^{-1}F}\Bigr]\mu(dz)\,. \label{gbg2}
 \end{align}
 If, furthermore, $F=I_q(f)$ for some $q\geq 1$ and some square-integrable, symmetric kernel $f$ on $\mathcal{Z}^q$ and $\E[F^2]=q!\Enorm{f}^2=2\nu$, then $-L^{-1}F=q^{-1}F$, 
\begin{align*}
 \E\bigl[\Gamma(F,-L^{-1}F)-2F\bigr]&=q^{-1}\E\bigl[\Gamma_0(F,F)\bigr]=2\nu\quad\text{and}\\
 \int_{\mathcal{Z}}\E\bigl[\abs{D_z^+F}^2\abs{D_z^+L^{-1}F}\bigr]\mu(dz)&=q^{-1}\int_{\mathcal{Z}}\E\bigl[\abs{D_z^+F}^3\bigr]\mu(dz)\\
 &\leq\biggl(\frac{2\nu}{q}\int_{\mathcal{Z}}\E\bigl[\abs{D_z^+F}^4\bigr]\mu(dz)\biggr)^{1/2}
\end{align*}
so that the previous estimate \eqref{gbg2} can be further bounded to give
\begin{align}
  d_2(F,Z_\nu)&\leq \max\Bigl(1,\frac{2}{\nu}\Bigr)\sqrt{\Var\Bigl(2F-q^{-1}\Gamma_0\bigl(F,F\bigr)\Bigr)}\notag\\
  &\;+\max\biggl(\sqrt{2\nu},\sqrt{\frac{2}{\nu}}+\sqrt{\frac{\nu}{2}}\,\biggr)\biggl(\frac{1}{q}\int_{\mathcal{Z}}\E\bigl[\abs{D_z^+F}^4\bigr]\mu(dz)\biggr)^{1/2}\,.\label{sbg1}
\end{align}
\end{prop}

\begin{proof}
Using the recently obtained bounds on the solution to the centered Gamma Stein equation from Theorem 2.3 of \cite{DP16b}, it is easy to see that 
\begin{equation*}
 d_2(F,Z_\nu)\leq\sup_{\psi\in\F_{2,\nu}}\Babs{\E\bigl[2(F+\nu)\psi'(F)-F\psi(F)\bigr]}\,,
\end{equation*}
where $\F_{2,\nu}$ denotes the class of all continuously differentiable functions $\psi$ in $\R$ such that both $\psi$ and $\psi'$ are Lipschitz-continuous with minimum Lipschitz-constants
\begin{equation*}
 \fnorm{\psi'}\leq\max\Bigl(1,\frac{2}{\nu}\Bigr)\quad\text{and}\quad \fnorm{\psi''}\leq \max\Bigl(2,\frac{1}{\nu}+1\Bigr)\,.
\end{equation*}
The rest of the argument follows a route that is completely analogous to the one leading to the proof of Proposition \ref{genbound}; the details are omitted for the sake of conciseness.\\
\end{proof}

\begin{lemma}\label{vargamma2}
Let $q\geq 1$ be an integer and and consider a random variable $F$ such that $F = I_q(f)\in C_q = {\rm Ker}(L+qI)$, {{} $\E[F^2]=2\nu$} and $\E[F^4]<\infty$. Then, $F, F^2\in {\rm dom}\, L$, and
\begin{align*}
 \Var\Bigl(2F-q^{-1}\Gamma\bigl(F,F\bigr)\Bigr)&=\sum_{\substack{1\leq p\leq 2q-1:\\p\not=q}}\Bigl(1-\frac{p}{2q}\Bigr)^2 \Var\bigl(\proj{F^2}{p}\bigr)\notag\\
 &\;+\frac14\Var\Bigl(\proj{F^2}{q}-4F\Bigr) \\
 &=\sum_{\substack{1\leq p\leq 2q-1:\\p\not=q}}\Bigl(1-\frac{p}{2q}\Bigr)^2 \Var\bigl(\proj{F^2}{p}\bigr)\notag\\
 &\;+\frac14\Var\bigl(\proj{F^2}{q}\bigr)+8\nu-2\E\bigl[F^3\bigr]=V_1+ V_2\notag\,,
\end{align*}
where we define 
\begin{align}
 V_1&:=\sum_{\substack{1\leq p\leq 2q-1:\\p\not=q}}\Bigl(1-\frac{p}{2q}\Bigr)^2 \Var\bigl(\proj{F^2}{p}\bigr)\quad\text{and}\label{defv1} \\
 V_2&:=\frac14\Var\bigl(\proj{F^2}{q}\bigr)+8\nu-2\E\bigl[F^3\bigr]=\frac14\Var\Bigl(\proj{F^2}{q}-4F\Bigr)\,.\label{defv2}
\end{align}

\end{lemma}

\begin{proof}
 The first identity easily follows from \eqref{vg1} and the orthogonality of the chaos decomposition. The second one follows from this and the formula
 \begin{equation*}
  \Var(X+Y)=\Var(X)+\Var(Y)+2\Cov(X,Y)
 \end{equation*}
upon observing that 
\begin{equation*}
 \Cov\Bigl(\proj{F^2}{q}\,, -4F\Bigr)=-4\E\bigl[F^3\bigr]\,,
\end{equation*}
again by orthogonality.\\
\end{proof}

\begin{lemma}\label{sandwich}
 Let $q\geq 1$ be an integer and let $F\in L^4(\Prob)$ be an element of the $q$-th Wiener chaos $C_q$, such that $F$ verifies {\bf Assumption A} and $\E[F^2] =2\nu$. The following relations are in order:
 \begin{align*}
 &\frac{1}{6q}\Bigl(\E\bigl[F^4\bigr]-12\E\bigl[F^3\bigr]-12\nu^2+48\nu\Bigr)+\frac{1}{12q^2}\int_{\mathcal{Z}}\E\bigl[\abs{D_z^+F}^4\bigr]\mu(dz)\notag\\
 &\leq \Var\Bigl(2F-q^{-1}\Gamma\bigl(F,F\bigr)\Bigr)= \Var\Bigl(2F-q^{-1}\Gamma_0\bigl(F,F\bigr)\Bigr)\notag\\
 &\leq\frac13\Bigl(\E\bigl[F^4\bigr]-12\E\bigl[F^3\bigr]-12\nu^2+48\nu\Bigr)+\frac{1}{6q}\int_{\mathcal{Z}}\E\bigl[\abs{D_z^+F}^4\bigr]\mu(dz)
\end{align*}
 \end{lemma}

 \begin{proof}
Recall that, under the assumptions in the statement, $\Gamma(F,F) = \Gamma_0(F,F)$. Using orthogonality, from Lemma \ref{remlemma} and \eqref{vg1} we obtain
\begin{align*}
 \E\bigl[F^4\bigr]&=\frac{3}{q}\E\bigl[F^2\Gamma(F,F)\bigr]-\frac{1}{2q}\int_{\mathcal{Z}}\E\bigl[\abs{D_z^+F}^4\bigr]\mu(dz)\notag\\
 &=3\bigl(\E[F^2]\bigr)^2+3\sum_{p=1}^{2q}\Bigl(1-\frac{p}{2q}\Bigr) \Var\bigl(\proj{F^2}{p}\bigr)-\frac{1}{2q}\int_{\mathcal{Z}}\E\bigl[\abs{D_z^+F}^4\bigr]\mu(dz)\notag\\
 &=12\nu^2+3\sum_{\substack{1\leq p\leq 2q-1:\\p\not=q}}\Bigl(1-\frac{p}{2q}\Bigr) \Var\bigl(\proj{F^2}{p}\bigr)+\frac{3}{2}\Var\bigl(\proj{F^2}{q}\bigr)\notag\\
 &\;-\frac{1}{2q}\int_{\mathcal{Z}}\E\bigl[\abs{D_z^+F}^4\bigr]\mu(dz)\,.
\end{align*}
Hence, recalling the definition of $V_2$ in \eqref{defv2} we conclude from Lemma \ref{vargamma2} that
\begin{align*}
 &\E\bigl[F^4\bigr]-12\E\bigl[F^3\bigr]-12\nu^2+48\nu\notag\\
 &=3\sum_{\substack{1\leq p\leq 2q-1:\\p\not=q}}\Bigl(1-\frac{p}{2q}\Bigr) \Var\bigl(\proj{F^2}{p}\bigr)-\frac{1}{2q}\int_{\mathcal{Z}}\E\bigl[\abs{D_z^+F}^4\bigr]\mu(dz)\notag\\
&\; +\frac{3}{2}\Var\bigl(\proj{F^2}{q}\bigr)-12\E\bigl[F^3\bigr]+48\nu\notag\\
&=3\sum_{\substack{1\leq p\leq 2q-1:\\p\not=q}}\Bigl(1-\frac{p}{2q}\Bigr) \Var\bigl(\proj{F^2}{p}\bigr)+6 V_2-\frac{1}{2q}\int_{\mathcal{Z}}\E\bigl[\abs{D_z^+F}^4\bigr]\mu(dz)\,.
\end{align*}
Now, recalling also the definition \eqref{defv1} of $V_1$ and using the simple chain of inequalities
\begin{equation*}
 \Bigl(1-\frac{p}{2q}\Bigr)^2\leq \Bigl(1-\frac{p}{2q}\Bigr)\leq 2q\Bigl(1-\frac{p}{2q}\Bigr)^2\,,\quad 1\leq p\leq 2q-1\,,
\end{equation*}
we obtain on the one hand that 
\begin{align}\label{sw1}
 &\E\bigl[F^4\bigr]-12\E\bigl[F^3\bigr]-12\nu^2+48\nu\notag\\
 &\geq3\sum_{\substack{1\leq p\leq 2q-1:\\p\not=q}}\Bigl(1-\frac{p}{2q}\Bigr)^2 \Var\bigl(\proj{F^2}{p}\bigr)+6 V_2-\frac{1}{2q}\int_{\mathcal{Z}}\E\bigl[\abs{D_z^+F}^4\bigr]\mu(dz)\notag\\
&=3V_1+6V_2-\frac{1}{2q}\int_{\mathcal{Z}}\E\bigl[\abs{D_z^+F}^4\bigr]\mu(dz)\notag\\
&\geq 3  \Var\Bigl(2F-q^{-1}\Gamma\bigl(F,F\bigr)\Bigr)-\frac{1}{2q}\int_{\mathcal{Z}}\E\bigl[\abs{D_z^+F}^4\bigr]\mu(dz)\,,
\end{align}
and, on the other hand, 
\begin{align}\label{sw2}
 &\E\bigl[F^4\bigr]-12\E\bigl[F^3\bigr]-12\nu^2+48\nu\notag\\
 &\leq 6q\sum_{\substack{1\leq p\leq 2q-1:\\p\not=q}}\Bigl(1-\frac{p}{2q}\Bigr)^2 \Var\bigl(\proj{F^2}{p}\bigr)+6 V_2-\frac{1}{2q}\int_{\mathcal{Z}}\E\bigl[\abs{D_z^+F}^4\bigr]\mu(dz)\notag\\
 &\leq 6q  \Var\Bigl(2F-q^{-1}\Gamma\bigl(F,F\bigr)\Bigr)-\frac{1}{2q}\int_{\mathcal{Z}}\E\bigl[\abs{D_z^+F}^4\bigr]\mu(dz)\,.
 \end{align}
The statement of the Lemma now follows from \eqref{sw1} and \eqref{sw2}.\\
\end{proof}

\begin{proof}[End of the proof of Theorem \ref{4mtg}] 
{ The claim of Theorem \ref{4mtg} is now an immediate consequence of the bound \eqref{sbg1} and of the upper bound given in Lemma \ref{sandwich}. }
\end{proof}

\section{Proofs of technical lemmas}\label{proofs}
 \subsection{ {}{Proof of Lemma \ref{Difflemma}}}
 
We first prove part (a). We just prove \eqref{dp2} and \eqref{dp3}, since the derivation of \eqref{dm2} and \eqref{dm3} is very similar. Let $f$ be a representative for $F$, i.e. $F=f(\eta)$. Then, by the binomial identity, we have 
 \begin{align*}
  \bigl(D_z^+F\bigr)^2&=\bigl(f(\eta+\delta_z)-f(\eta)\bigr)^2=f(\eta+\delta_z)^2-f(\eta)^2-2f(\eta+\delta_z) f(\eta)+2f(\eta)^2\notag\\
  &=D_z^+F^2-2f(\eta)\bigl(f(\eta+\delta_z)- f(\eta)\bigr)=D_z^+F^2-2 F D_z^+F
 \end{align*}
such that \eqref{dp2} holds true. Similarly, using \eqref{dp2}, we obtain
\begin{align*}
 \bigl(D_z^+F\bigr)^3&=\bigl(f(\eta+\delta_z)-f(\eta)\bigr)^3=f(\eta+\delta_z)^3-f(\eta)^3-3f(\eta+\delta_z)^2f(\eta)\notag\\
 &\;+3f(\eta+\delta_z)f(\eta)^2\notag\\
 &=D_z^+F^3+3f(\eta)^2\bigl(f(\eta+\delta_z)- f(\eta)\bigr)-3f(\eta)\bigl(f(\eta+\delta_z)^2- f(\eta)^2\bigr)\notag\\
&=D_z^+F^3+3F^2 D_z^+F-3FD_z^+F^2\notag\\
&=D_z^+F^3+3F^2 D_z^+F-3F\bigl(D_z^+F\bigr)^2-6F^2 D_z^+F\notag\\
&=D_z^+F^3-3F^2 D_z^+F-3F\bigl(D_z^+F\bigr)^2
 \end{align*}
which is equivalent to \eqref{dp3}. Now we turn to the proof of (b). Again, we just prove the part involving $D_z^+$. By a suitable version of Taylor's formula, for $x,y \in\R$ we have 
\begin{align*}
 \psi(y)=\psi(x)+\psi'(x)(y-x)+R_\psi(x,y)(y-x)^2\,,
\end{align*}
where 
\[\babs{R_\psi(x,y)}\leq\frac{\fnorm{\psi''}}{2}\,.\]
 Now the result follows by letting $x=F=f(\eta)$, $y=f(\eta+\delta_z)$ and $R_\psi^+(F,z)=R_\psi(f(\eta),f(\eta+\delta_z))$.

\qed

\subsection{{} Proof of Lemma \ref{chaosorder}}
{{} The method of proof we apply is similar to the one used for the proof of Proposition 5 in \cite{Lastsv}, which gives the product formula for multiple Wiener-It\^{o} integrals.}
Let 
\begin{equation}\label{chaosFG}
 FG=\E[FG]+\sum_{m=1}^\infty I_m(h_m)
\end{equation}
denote the chaos decomposition of $FG$. 
{ Let us, for the moment, only assume that $F=I_p(f), G=I_q(g)\in L^2(\Prob)$. We will prove the following 
statements $(\tilde{a})$ and $(\tilde{b})$ simultaneously by induction on $k:=p+q\geq2$:
\begin{align*}
{\text (\tilde{a})}&\quad \frac{1}{m!}\E\bigl[D^{(m)}_{z_1,\dotsc,z_m}(FG)\bigr]=0\quad\text{for all $m>k$ and $z_1,\dotsc,z_m\in\calZ$ }\\
{\text (\tilde{b})}&\quad \frac{1}{k!}\E\bigl[D^{(k)}_{z_1,\dotsc,z_k}(FG)\bigr]
=f\tilde{\otimes} g(z_1,\dotsc,z_k)\quad\text{for all $z_1,\dotsc,z_k\in\calZ$ }
\end{align*}
Once this is shown, if $F,G\in L^4(\Prob)$ such that $FG$ has a chaotic decomposition \eqref{chaosFG}, (a) and (b) immediately follow from $(\tilde{a})$ and $(\tilde{b})$, respectively, by virtue of \eqref{kerform}.
If $k=2$, then necessarily $p=q=1$ and, by \eqref{e:mix+} and since $F,G\in\dom D$, for all $y,z\in\calZ$ we have 
 \begin{align*}
  D^+_{z}(FG)&=f(z) I_1(g)+g(z)I_1(f)+f(z)g(z)\quad\text{and}\\
  D^{(2)}_{y,z}(FG)&=f(z)g(y)+f(y)g(z)=2f\tilde{\otimes} g(y,z)\,.
 \end{align*}
This immediately implies that $D^{(m)}(FG)=0$ for all $m>2$. We can thus infer that 
\begin{align*}
 \frac12\E\bigl[D^{(2)}_{z_1,z_2}(FG)\bigr]=f\tilde{\otimes} g(z_1,z_2)\quad\text{and}\quad\frac{1}{m!}\E\bigl[D^{(m)}_{z_1,\dotsc,z_m}(FG)\bigr]=0
\end{align*}
for all $m>2$ and $z_1,\dotsc,z_m\in\calZ$. Thus, $(\tilde{a})$ and $(\tilde{b})$ hold true for $k=2$.}
Now assume that $k>2$. { Again, since $F,G\in\dom D$, from \eqref{e:mix+} we obtain that }
\begin{align}\label{co1}
 D^+_{z_k}(FG)&=pI_q(g)I_{p-1}\bigl(f(z_k,\cdot)\bigr) +qI_p(f)I_{q-1}\bigl(g(z_k,\cdot)\bigr)\notag\\
 &\;+ pq I_{p-1}\bigl(f(z_k,\cdot)\bigr)I_{q-1}\bigl(g(z_k,\cdot)\bigr)\notag\\
 & =:p \tilde{F}_{z_k} G+q\tilde{G}_{z_k} F+pq \tilde{F}_{z_k} \tilde{G}_{z_k} 
\end{align}
holds for all $z_k\in\calZ$, where $\tilde{F}_{z_k}$ and $\tilde{G}_{z_k}$ are multiple integrals of orders $p-1$ and $q-1$, respectively. Hence, by the induction hypothesis for claim $(\tilde{a})$ we already conclude that
\begin{equation*}
 \E\bigl[D^{(k-1)}_{z_1,\dotsc,z_{k-1}}\bigl(\tilde{F}_{z_k}\tilde{G}_{z_k}\bigr)\bigr]=0\,.
\end{equation*}
so that 
\begin{equation*}
 \E\bigl[D^{(k)}_{z_1,\dotsc,z_k}(FG)\bigr]=p\E\Bigl[D^{(k-1)}_{z_1,\dotsc,z_{k-1}}\bigl(\tilde{F}_{z_k}G\bigr)\Bigr]
 +q\E\Bigl[D^{(k-1)}_{z_1,\dotsc,z_{k-1}}\bigl(F\tilde{G}_{z_k}\bigr)\Bigr]\,.
\end{equation*}
By the induction hypothesis for claim $(\tilde{b})$ we have 
\begin{align*}
 \E\Bigl[D^{(k-1)}_{z_1,\dotsc,z_{k-1}}\bigl(\tilde{F}_{z_k}G\bigr)\Bigr]&=(k-1)!\bigl(f(z_k,\cdot)\tilde{\otimes} g\bigr)(z_1,\dotsc,z_{k-1})\quad\text{and}\\
 \E\Bigl[D^{(k-1)}_{z_1,\dotsc,z_{k-1}}\bigl(F\tilde{G}_{z_k}\bigr)\Bigr]&=(k-1)!\bigl(f\tilde{\otimes}\bigl(g(z_k,\cdot)\bigr)(z_1,\dotsc,z_{k-1})
\end{align*}
and, in order to prove $(\tilde{b})$, it remains to show that 
\begin{align}\label{c02}
k! (f\tilde{\otimes} g)(z_1,\dotsc,z_{k})&= p(k-1)!\bigl(f(z_k,\cdot)\tilde{\otimes} g\bigr)(z_1,\dotsc,z_{k-1})\notag\\
&\;+q(k-1)!\bigl(f\tilde{\otimes}\bigl(g(z_k,\cdot)\bigr)(z_1,\dotsc,z_{k-1})\,.
\end{align}
This, however, follows from 
\begin{align*}
 k!(f\tilde{\otimes} g)(z_1,\dotsc,z_{k})&=\sum_{\pi\in\mathbb{S}_{p+q}} f(z_{\pi(1)},\dotsc,z_{\pi(p)})g(z_{\pi(p+1)},\dotsc,z_{\pi(p+q)})\\
 &=\sum_{\pi:k\in\{\pi(1),\dotsc,\pi(p)\}} f(z_{\pi(1)},\dotsc,z_{\pi(p)})g(z_{\pi(p+1)},\dotsc,z_{\pi(p+q)})\\
 &\;+\sum_{\pi:k\notin\{\pi(1),\dotsc,\pi(p)\}} f(z_{\pi(1)},\dotsc,z_{\pi(p)})g(z_{\pi(p+1)},\dotsc,z_{\pi(p+q)})\\
 &\stackrel{!}{=}p\sum_{\tau\in\mathbb{S}_{p+q-1}}f(z_k, z_{\tau(1)},\dotsc,z_{\tau(p-1)})g(z_{\tau(p)},\dotsc,z_{\tau(p+q-1)})\\
 &\;+q\sum_{\tau\in\mathbb{S}_{p+q-1}}f(z_{\tau(1)},\dotsc,z_{\tau(p)})g(z_k,z_{\tau(p+1)},\dotsc,z_{\tau(p+q-1)})\\
 &=p(k-1)!\bigl(f(z_k,\cdot)\tilde{\otimes} g\bigr)(z_1,\dotsc,z_{k-1})\notag\\
&\;+q(k-1)!\bigl(f\tilde{\otimes}\bigl(g(z_k,\cdot)\bigr)(z_1,\dotsc,z_{k-1})\,.
\end{align*}
We explain the identity involving $!$ in some more detail. Consider the first sum appearing there and note that 
\begin{align*}
&\sum_{\pi:k\in\{\pi(1),\dotsc,\pi(p)\}} f(z_{\pi(1)},\dotsc,z_{\pi(p)})g(z_{\pi(p+1)},\dotsc,z_{\pi(p+q)})\\
&=\sum_{j=1}^p \sum_{\pi:\pi(j)=k} f(z_{\pi(1)},\dotsc,z_{\pi(j-1)},z_k,z_{\pi(j+1)},\dotsc,z_{\pi(p)})g(z_{\pi(p+1)},\dotsc,z_{\pi(p+q)})\\
&=p\sum_{\pi:\pi(1)=k} f(z_k,z_{\pi(2)},\dotsc,z_{\pi(p)})g(z_{\pi(p+1)},\dotsc,z_{\pi(p+q)})
\end{align*}
where we have used the symmetry of the kernel $f$ to obtain the last identity.
Now, since the mapping 
\begin{align*}
\Psi:\mathbb{S}_{k-1}\rightarrow\{\pi\in\mathbb{S}_k\,:\,\pi(1)=k\}\,,
\quad \Psi(\sigma)(j):=\begin{cases}
k\,,& j=1\\
\sigma(j-1)\,, &j\in\{2,\dotsc,k\}
\end{cases}
\end{align*}
is a bijection, we obtain that
\begin{align*}
&\sum_{\pi:\pi(1)=k} f(z_k,z_{\pi(2)},\dotsc,z_{\pi(p)})g(z_{\pi(p+1)},\dotsc,z_{\pi(p+q)})\\
&=\sum_{\tau\in\mathbb{S}_{p+q-1}}f(z_k, z_{\tau(1)},\dotsc,z_{\tau(p-1)})g(z_{\tau(p)},\dotsc,z_{\tau(p+q-1)})
\end{align*}
proving the claim. Thus, we have proved $(\tilde{b})$.\\
If $m>k$ and $z_1,\dotsc,z_m\in\calZ$, then, by the induction hypothesis on $(\tilde{a})$ and from \eqref{co1} we obtain
{
\begin{align*}
 \E\bigl[D^{(m)}_{z_1,\dotsc,z_m}(FG)\bigr] &=p \E\Bigl[D^{(m-1)}_{z_1,\dotsc,z_{m-1}}\bigl(\tilde{F}_{z_m} G\bigr)\Bigr]
 +q \E\Bigl[D^{(m-1)}_{z_1,\dotsc,z_{m-1}}\bigl(F\tilde{G}_{z_m} \bigr)\Bigr]\\
&\hspace{2cm} +pq\E\Bigl[D^{(m-1)}_{z_1,\dotsc,z_{m-1}}\bigl(\tilde{F}_{z_m}\tilde{G}_{z_m} \bigr)\Bigr]=0
 \end{align*}
}
for all $z_1,\dotsc,z_m\in\calZ$, proving $(\tilde{a})$.
\qed

{{} \subsection{Stein's equation in the Kolmogorov distance}\label{ss:stein}

In order to deal with bounds in the Kolmogorov distance involving remove-one cost operators, we need the following result, containing several estimates on the solution of the Stein's equation associated with test functions having the form of indicators of half-lines. Points (a)-(f) are well-known. Point (g) is standard but not explicitly stated in the literature (to our knowledge) --- a proof is provided for the sake of completeness.

\begin{prop}\label{p:stein} Let $N\sim {N}(0,1)$ be a centred Gaussian random variable with unit variance and, for every $x \in\mathbb{R}$, introduce the Stein's equation
\begin{equation}
\label{eq:stein-equation}
g'(w)-wg(w)={\bf 1}_{\{w\leqslant x\}}-\Prob(N\leqslant x),
\end{equation}
where $w\in \mathbb{R}$. Then, for every real $x$, there exists a function $g_x : \mathbb{R}\to \mathbb{R} : w\mapsto g_x(w)$ satisfying the following properties {\rm (a)-(g)}:
\begin{itemize}

\item[\rm (a)] $g_x$ is continuous at every point $w\in \mathbb{R}$, and infinitely differentiable at every $w\neq x$;

\item[\rm (b)] $g_x$ satisfies the relation \eqref{eq:stein-equation}, for every $w\neq x$;

\item[\rm (c)] $0<g_{x}\leqslant  \frac{\sqrt{2\pi }}{4}$;

\item[\rm (d)] for every $u,v,w\in \mathbb{R}$,
\begin{equation}\label{e:zumba}
\vert (w+u)g_x(w+u) - (w+v)g_x(w+v)\vert \leqslant \left( \vert w \vert + \frac{\sqrt{2\pi}}{4}\right) \left(\vert u \vert + \vert v \vert \right);
\end{equation}

\item[\rm (e)] adopting the convention 
\begin{equation}\label{e:convention}
g'_x(x) : = xg_x(x)+1-\Prob(N\leqslant x), \quad 
\end{equation}
one has that $|g'_x(w)|\leqslant 1$, for every real $w$ ;

\item[\rm (f)] using again the convention \eqref{e:convention}, for all $w,h\in \mathbb{R}$ one has that
\begin{align}
\label{eq:stein}
 | g_{x}(w+h)-g_{x}(w)-g_{x}'(w)h | &\leqslant \frac{ |  h|^{2} }{2}\left(  | w| +\frac{\sqrt{2\pi }}{4} \right)\\
 &\quad\quad\quad + | h | ({\bf 1}_{[w, w+h)}(x) + {\bf 1}_{[w+h, w)}(x))\notag \\
 &= \frac{ |  h|^{2} }{2}\left(  | w| +\frac{\sqrt{2\pi }}{4} \right)\label{eq:stein2} \\
 &\quad\quad\quad+ h\left({\bf 1}_{ \left[ w , w + h \right)}(x) - {\bf 1}_{\left[  w+h , w\right)}(x) \right); \notag
\end{align}

\item[\rm (g)] under \eqref{e:convention}, for every $w,h\in \mathbb{R}$ one has that
\begin{align}
\label{eq:steinm}
 | g_x(w)-g_x(w-h)-g_x'(w)h | &\leqslant \frac{ 3 |  h|^{2} }{2}\left(  | w-h | +\frac{\sqrt{2\pi }}{4} \right)\\
 &\quad\quad\quad + | h | ({\bf 1}_{[w-h, w)}(x) + {\bf 1}_{[w, w-h)}(x))\notag \\
 &= \frac{ 3|  h|^{2} }{2}\left(  | w-h| +\frac{\sqrt{2\pi }}{4} \right)\label{eq:steinm2} \\
 &\quad\quad\quad+ h\left({\bf 1}_{ \left[ w-h , w \right)}(x) - {\bf 1}_{\left[  w , w-h\right)}(x) \right).\notag
\end{align}

\end{itemize}
\end{prop}

\begin{proof} The content of Points (a)--(f) is well-known -- see e.g. \cite[Section 2.2.2]{BPsv} and the references therein. To show (g), fix $x\in \mathbb{R}$, recall \eqref{e:convention} and write, for every $w,h \in \mathbb{R}$, \begin{equation*}
g_x(w)-g_x(w-h) - hg'_x(w) = \int_{0}^{h}\left(g_x'(w-h+u) - g_x'(w) \right) du.
\end{equation*}  
Since $g_x$ is a solution of \eqref{eq:stein-equation} for every real $w$, we have that, for all $w,h \in \mathbb{R}$,
\begin{eqnarray*}
&& g_x(w)-g_x(w-h) - hg'_x(w) \\
&& = \int_{0}^{h}\bigl((w-h+u)g_x(w-h+u) - wg_x(w) \bigr) du + \int_{0}^{h}\left({\bf 1}_{\left\lbrace w-h + u \leqslant x\right\rbrace } - {\bf 1}_{\left\lbrace w \leqslant x\right\rbrace } \right) du \\ &&:= J_1 + J_2.
\end{eqnarray*} 
It follows that, by the triangle inequality,
\begin{equation}
\label{boundingthediff}
\left| g_x(w)-g_x(w-h) - hg'_x(x)\right| \leqslant \vert J_1 \vert + \vert J_2 \vert.
\end{equation} 
Using \eqref{e:zumba}, we have 
\begin{equation}
\label{boundI1}
\vert J_1 \vert \leqslant \int_{0}^{h}\left(|w-h| +\frac{\sqrt{2\pi}}{4}\right)(\vert u \vert + \vert h\vert) du = \frac{3h^2}{2}\left(|w-h| +\frac{\sqrt{2\pi}}{4}\right).
\end{equation} 
On the other hand, we have that
\begin{eqnarray*}
\vert J_2 \vert &=& {\bf 1}_{\left\lbrace h <0\right\rbrace }\left| \int_{0}^{h}\left({\bf 1}_{\left\lbrace w-h + u \leqslant x\right\rbrace } - {\bf 1}_{\left\lbrace w \leqslant x\right\rbrace } \right) du \right| \\ &&\quad\quad\quad\quad \quad\quad+ {\bf 1}_{\left\lbrace h \geq 0\right\rbrace }\left| \int_{0}^{h}\left({\bf 1}_{\left\lbrace w-h + u \leqslant x\right\rbrace } - {\bf 1}_{\left\lbrace w \leqslant x\right\rbrace } \right) du \right| \\
&=& {\bf 1}_{\left\lbrace h <0\right\rbrace }\int_{h}^{0}{\bf 1}_{\left\lbrace w \leqslant x < w-h+u\right\rbrace }du  + {\bf 1}_{\left\lbrace h \geq 0\right\rbrace }\int_{0}^{h}{\bf 1}_{\left\lbrace w-h+u \leqslant  x< w \right\rbrace } du.
\end{eqnarray*}
As a consequence,
\begin{eqnarray}
\vert J_2 \vert & \leqslant &\!\! {\bf 1}_{\left\lbrace h <0\right\rbrace }(-h){\bf 1}_{\left[  w , w-h\right)}(x)  + {\bf 1}_{\left\lbrace h \geq 0\right\rbrace }h{\bf 1}_{ \left[ w-h , w  \right)}(x)\label{boundI2} \\
& = &\!\! h\left({\bf 1}_{ \left[w-h , w  \right)}(x) - {\bf 1}_{\left[  w , w-h\right)}(x) \right)\notag \!\!=\!\! \vert h \vert \left({\bf 1}_{ \left[ w-h , w  \right)}(x) + {\bf 1}_{\left[  w , w-h\right)}(x) \right).
\end{eqnarray}
Using \eqref{boundI1} and \eqref{boundI2} in \eqref{boundingthediff} yields the conclusion.
\end{proof}

}

\normalem
\bibliography{poisson}{}

\begin{thebibliography}{AMMP16}

\bibitem[ACP14]{ACP}
E.~Azmoodeh, S.~Campese, and G.~Poly.
\newblock Fourth {M}oment {T}heorems for {M}arkov diffusion generators.
\newblock {\em J. Funct. Anal.}, 266(4):2341--2359, 2014.

\bibitem[AMMP16]{AMMP}
E.~Azmoodeh, D.~Malicet, G.~Mijoule, and G.~Poly.
\newblock Generalization of the nualart-peccati criterion.
\newblock {\em Ann. Probab.}, 44(2):924--954, 2016.

\bibitem[BD15]{BDbook}
N.~Bouleau and L.~Denis.
\newblock {\em Dirichlet Forms Methods for Poisson Point Measures and L\'evy
  Processes}.
\newblock Probability Theory and Stochastic Modelling. Springer Verlag, 2015.

\bibitem[BGL14]{BGL14}
D.~Bakry, I.~Gentil, and M.~Ledoux.
\newblock {\em Analysis and geometry of {M}arkov diffusion operators}, volume
  348 of {\em Grundlehren der Mathematischen Wissenschaften [Fundamental
  Principles of Mathematical Sciences]}.
\newblock Springer, Cham, 2014.

\bibitem[BP14]{BPjfa}
S.~Bourguin and G.~Peccati.
\newblock Semicircular limits on the free poisson chaos: counterexamples to a
  transfer principle.
\newblock {\em J. Funct. Anal.}, 267(4):963--997, 2014.

\bibitem[BP16a]{BacP}
S.~Bachmann and G.~Peccati.
\newblock Concentration bounds for geometric poisson functionals: logarithmic
  sobolev inequalities revisited.
\newblock {\em Electron. J. Probab.}, 21:no. 6, 21, 2016.

\bibitem[BP16b]{BPsv}
S.~Bourguin and G.~Peccati.
\newblock {The Malliavin-Stein method on the Poisson space}.
\newblock In G.~Peccati and M.~Reitzner, editors, {\em {Stochastic analysis for
  Poisson point processes}}, Mathematics, Statistics, Finance and Economics,
  chapter~6, pages 185--228. Bocconi University Press and Springer, 2016.

\bibitem[CGS11]{CGS}
L.~H.~Y. Chen, L.~Goldstein, and Q.-M. Shao.
\newblock {\em Normal approximation by {S}tein's method}.
\newblock Probability and its Applications (New York). Springer, Heidelberg,
  2011.

\bibitem[CP15]{ChPo}
L.~H.~Y. Chen and G.~Poly.
\newblock Stein's method, {M}alliavin calculus, {D}irichlet forms and the
  fourth moment theorem.
\newblock In {\em Festschrift {M}asatoshi {F}ukushima}, volume~17 of {\em
  Interdiscip. Math. Sci.}, pages 107--130. World Sci. Publ., Hackensack, NJ,
  2015.

\bibitem[dJ87]{deJo87}
P.~de~Jong.
\newblock A central limit theorem for generalized quadratic forms.
\newblock {\em Probab. Theory Related Fields}, 75(2):261--277, 1987.

\bibitem[dJ89]{deJo89}
P.~de~Jong.
\newblock {\em Central limit theorems for generalized multilinear forms},
  volume~61 of {\em CWI Tract}.
\newblock Stichting Mathematisch Centrum, Centrum voor Wiskunde en Informatica,
  Amsterdam, 1989.

\bibitem[dJ90]{deJo90}
P.~de~Jong.
\newblock A central limit theorem for generalized multilinear forms.
\newblock {\em J. Multivariate Anal.}, 34(2):275--289, 1990.

\bibitem[DK17]{DK17}
C.~D\"obler and K.~Krokowski.
\newblock {On the fourth moment condition for Rademacher chaos}.
\newblock {\em {\tt arXiv:1706.00751}}, 2017.

\bibitem[DP17a]{DP16}
C.~D\"obler and G.~Peccati.
\newblock {Quantiative de Jong theorems in any dimension}.
\newblock {\em Electron. J. Probab.}, 22:no. 2, 1--35, 2017.

\bibitem[DP17b]{DP16b}
C.~D\"obler and G.~Peccati.
\newblock {The Gamma Stein equation and non-central de Jong theorems}.
\newblock {\em {\tt arXiv:1612.02279}, to appear in: Bernoulli}, 2017.

\bibitem[DVZ17]{DVZ}
C.~D\"obler, A.~Vidotto, and G.~Zheng.
\newblock {Fourth moment theorems on the Poisson space in any dimension}.
\newblock {\em {\tt arXiv:1707.01889}}, 2017.

\bibitem[ET14]{ET14}
P.~Eichelsbacher and C.~Th{\"a}le.
\newblock New {B}erry-{E}sseen bounds for non-linear functionals of {P}oisson
  random measures.
\newblock {\em Electron. J. Probab.}, 19:no. 102, 25, 2014.

\bibitem[FT16]{FT}
T.~Fissler and C.~Th{\"a}le.
\newblock A four moments theorem for gamma limits on a {P}oisson chaos.
\newblock {\em ALEA Lat. Am. J. Probab. Math. Stat.}, 13(1):163--192, 2016.

\bibitem[Las16]{Lastsv}
G.~Last.
\newblock {Stochastic analysis for Poisson processes}.
\newblock In G.~Peccati and M.~Reitzner, editors, {\em {Stochastic analysis for
  Poisson point processes}}, Mathematics, Statistics, Finance and Economics,
  chapter~1, pages 1--36. Bocconi University Press and Springer, 2016.

\bibitem[Led12]{Led12}
M.~Ledoux.
\newblock Chaos of a {M}arkov operator and the fourth moment condition.
\newblock {\em Ann. Probab.}, 40(6):2439--2459, 2012.

\bibitem[LNP15]{LNP}
M.~Ledoux, I.~Nourdin, and G.~Peccati.
\newblock Stein's method, logarithmic sobolev and transport inequalities.
\newblock {\em Geom. Funct. Anal.}, 25:256--306, 2015.

\bibitem[LP11]{LaPen11}
G.~Last and M.~D. Penrose.
\newblock Poisson process {F}ock space representation, chaos expansion and
  covariance inequalities.
\newblock {\em Probab. Theory Related Fields}, 150(3-4):663--690, 2011.

\bibitem[LP17]{LPbook}
G.~Last and M.~Penrose.
\newblock {\em Lectures on the Poisson Process}.
\newblock IMS Textbooks. Cambridge University Press, Cambridge, 2017.

\bibitem[LPS16]{LPS}
G.~Last, G.~Peccati, and M.~Schulte.
\newblock Normal approximation on {P}oisson spaces: {M}ehler's formula, second
  order {P}oincar\'e inequalities and stabilization.
\newblock {\em Probab. Theory Related Fields}, 165(3-4):667--723, 2016.

\bibitem[LRP13]{LRP1}
R.~Lachi{{\`e}}ze-Rey and G.~Peccati.
\newblock Fine {G}aussian fluctuations on the {P}oisson space, {I}:
  contractions, cumulants and geometric random graphs.
\newblock {\em Electron. J. Probab.}, 18:no. 32, 32, 2013.

\bibitem[LRR16]{LRRsv}
R.~Lachi{{\`e}}ze-Rey and M.~Reitzner.
\newblock {$U$-statistics in stochastic geometry}.
\newblock In G.~Peccati and M.~Reitzner, editors, {\em {Stochastic analysis for
  Poisson point processes}}, Mathematics, Statistics, Finance and Economics,
  chapter~7, pages 229--253. Bocconi University Press and Springer, 2016.

\bibitem[LRSY16]{LSY16}
R.~Lachi\`eze-Rey, M.~Schulte, and J.~Yukich.
\newblock Normal approximation for sums of stabilizing functionals.
\newblock {\em Preprint available at: {\tt
  http://www.lehigh.edu/~jey0/LSYDec17-2016.pdf}}, 2016.

\bibitem[Mec67]{Mecke}
J.~Mecke.
\newblock Station\"are zuf\"allige {M}asse auf lokalkompakten {A}belschen
  {G}ruppen.
\newblock {\em Z. Wahrscheinlichkeitstheorie und Verw. Gebiete}, 9:36--58,
  1967.

\bibitem[NP05]{NP05}
D.~Nualart and G.~Peccati.
\newblock Central limit theorems for sequences of multiple stochastic
  integrals.
\newblock {\em Ann. Probab.}, 33(1):177--193, 2005.

\bibitem[NP09a]{NP09a}
I.~Nourdin and G.~Peccati.
\newblock Noncentral convergence of multiple integrals.
\newblock {\em Ann. Probab.}, 37(4):1412--1426, 2009.

\bibitem[NP09b]{NP-ptrf}
I~Nourdin and G.~Peccati.
\newblock Stein's method on wiener chaos.
\newblock {\em Probab. Theory Related Fields}, 145(1):75--118, 2009.

\bibitem[NP12]{NouPecbook}
I.~Nourdin and G.~Peccati.
\newblock {\em Normal approximations with {M}alliavin calculus}, volume 192 of
  {\em Cambridge Tracts in Mathematics}.
\newblock Cambridge University Press, Cambridge, 2012.
\newblock From Stein's method to universality.

\bibitem[PR16]{PecRei16}
G.~Peccati and M.~Reitzner.
\newblock {\em {Stochastic Analysis for Poisson Point Processes}}.
\newblock Mathematics, Statistics, Finance and Economics. Bocconi University
  Press and Springer, 2016.

\bibitem[PSTU10]{PSTU}
G.~Peccati, J.~L. Sol{{\'e}}, M.~S. Taqqu, and F.~Utzet.
\newblock Stein's method and normal approximation of {P}oisson functionals.
\newblock {\em Ann. Probab.}, 38(2):443--478, 2010.

\bibitem[PT08]{PT08}
G.~Peccati and M.~S. Taqqu.
\newblock Central limit theorems for double {P}oisson integrals.
\newblock {\em Bernoulli}, 14(3):791--821, 2008.

\bibitem[PT13]{PTh}
G.~Peccati and C.~Th{{\"a}}le.
\newblock Gamma limits and {$U$}-statistics on the {P}oisson space.
\newblock {\em ALEA Lat. Am. J. Probab. Math. Stat.}, 10(1):525--560, 2013.

\bibitem[PZ10]{PZ1}
G.~Peccati and C.~Zheng.
\newblock Multi-dimensional {G}aussian fluctuations on the {P}oisson space.
\newblock {\em Electron. J. Probab.}, 15:no. 48, 1487--1527, 2010.

\bibitem[PZ14]{PZ2}
G.~Peccati and C.~Zheng.
\newblock Universal {G}aussian fluctuations on the discrete {P}oisson chaos.
\newblock {\em Bernoulli}, 20(2):697--715, 2014.

\bibitem[RS13]{RS}
M.~Reitzner and M.~Schulte.
\newblock {Central limit theorems for U-statistics of Poisson point processes}.
\newblock {\em Ann. Probab.}, 41(6):3879--3909, 2013.

\bibitem[Sch16]{Sch16}
M.~Schulte.
\newblock Normal approximation of {P}oisson functionals in {K}olmogorov
  distance.
\newblock {\em J. Theoret. Probab.}, 29(1):96--117, 2016.

\bibitem[ST16]{STsv}
M.~Schulte and C.~Th\"ale.
\newblock {Poisson Point Process Convergence and Extreme Values in Stochastic
  Geometry}.
\newblock In G.~Peccati and M.~Reitzner, editors, {\em {Stochastic analysis for
  Poisson point processes}}, Mathematics, Statistics, Finance and Economics,
  chapter~7, pages 255--294. Bocconi University Press and Springer, 2016.

\end{thebibliography}
\bibliographystyle{alpha}
\end{document}